\documentclass{amsart}
\usepackage{amsmath,amscd,amsthm,amssymb,amsxtra,latexsym,epsfig,epic,graphics}
\usepackage[all]{xy}
\usepackage{MnSymbol}
\usepackage{color} 
\usepackage{booktabs}
\usepackage{multirow, color}
\usepackage{graphicx} 
\usepackage{amssymb}
\usepackage{amsrefs}
\usepackage{hyperref}
\usepackage{comment}
\usepackage{xcolor}
\usepackage{cite}
\usepackage{hyperref}

\usepackage{color}
\usepackage[all]{xy}
\newtheorem{theorem}{Theorem}[section]

\newtheorem{proposition}[theorem]{Proposition}

\theoremstyle{definition}

\newtheorem{example}[theorem]{Example}

\newtheorem{remark}[theorem]{Remark}

\newtheorem{question}[theorem]{Question}
\newtheorem{problem}[theorem]{Problem}

\newcommand{\Pn}{{\mathbb P}}
\newcommand{\PP}{{\mathbb P}}

\newcommand{\CC}{{\mathbb C}}
\newcommand{\ZZ}{{\mathbb Z}}
\newcommand{\QQ}{{\mathbb Q}}

\newcommand{\UU}{{\mathbb U}}

\newcommand{\FF}{{\mathbb F}}
\newcommand{\MM}{{\mathbb M}}
\newcommand{\GG}{{\mathbb G}}
\newcommand{\TT}{{\mathbb T}}

\newcommand{\cO}{\mathcal{O}}

\newcommand{\sO}{\mathcal{O}}
\newcommand{\sI}{\mathcal{I}}
\newcommand{\sF}{\mathcal{F}}
\newcommand{\sG}{\mathcal{G}}
\newcommand{\sE}{\mathcal{E}}
\newcommand{\sL}{\mathcal{L}}
\newcommand{\PGL}{\rm {PGL}}

\DeclareMathOperator{\SL}{SL}

\DeclareMathOperator{\rank}{rk}

\DeclareMathOperator{\coker}{coker}

\DeclareMathOperator{\Proj}{Proj}
\DeclareMathOperator{\Pic}{Pic}

\DeclareMathOperator{\Hom}{Hom}

\DeclareMathOperator{\Tor}{Tor}

\DeclareMathOperator{\Hilb}{Hilb}
\DeclareMathOperator{\Spec}{Spec}

\DeclareMathOperator{\syz}{syz}
\DeclareMathOperator{\h}{h}

\title{Nongeneral type surfaces in $\Pn^4$, an update}
\author[Hirotachi Abo]{Hirotachi Abo}
\address{H. Abo: Department of Mathematics and Statistical Science, University of Idaho, Moscow, Idaho 83844-1103, United States of
America}
\email{abo@uidaho.edu}
\author[Kristian Ranestad]{Kristian Ranestad}
\address{K.Ranestad: Matematisk institutt\\
         Universitetet i Oslo\\
         PO Box 1053, Blindern\\
         NO-0316 Oslo\\
         Norway}
\email{ranestad@math.uio.no}
\author[Frank-Olaf Schreyer]{Frank-Olaf Schreyer}
\address{F.-O. Schreyer: Mathematik und Informatik, Universit\"at des Saarlandes, 66123 Saarbr\"ucken, Germany}
\email{schreyer@math.uni-sb.de}
\date{July 2026}


\begin{document}

\begin{abstract}
A general algebraic surface cannot be embedded in $\Pn^4$.  Proving a conjecture by Hartshorne and Lichtenbaum, Ellingsrud and Peskine showed that there is a degree bound for smooth rational surfaces in $\Pn^4$, and in fact for surfaces not of general type.  We give a survey of the classification status and of classical and computer-aided constructions of smooth nongeneral-type surfaces in $\Pn^4$. \end{abstract}

\maketitle

\tableofcontents

\section{Introduction}

Smooth surfaces in $\PP^3$ are rather special from a birational point of view: A smooth surface $X \subset \PP^3$ of degree $d$ is rational if $d\le 3$, 
a K3 surface  if  $d= 4$, and
a minimal surface of general type  if  $d\ge 5$.

On the other hand,
every smooth complex projective surface $X$ can be embedded in $\Pn^5$ by first embedding $X\subset \Pn^N$ for some $N$ and then, if $N>5$, projecting from a general linear space $L$ of codimension six to $\Pn^5$: The secant and tangent lines to $X$ in $\Pn^N$ fill a variety of dimension at most five that does not intersect $L$, so the projection of $X$ from $L$ is an embedding. A further projection to $\Pn^4$ will typically introduce singularities.

The numerical invariants of a smooth surface $X\subset \PP^4$ satisfy the \textbf{double point formula}
\cite[App A, Ex 4.1.3]{Har77}:
    \[d^2-10d-5H\cdot K_X-2K_X^2+12{\chi}(X) = 0\]
where $H$ denotes the hyperplane class, $d=H^2$ the degree, $K_X$ the canonical class, and $\chi(X)=\chi(\sO_X)$
the holomorphic Euler characteristic.
Equivalently,
    \[d^2-5d-5(2\pi(X)-2)-2K_X^2+12{\chi}(X) = 0\]
since the sectional genus $\pi(X)$ satisfies $2\pi(X)-2=H.(H+K_X)$, by the adjunction formula.
For a surface in $\PP^5$, the left-hand side of the formula gives the number of non-Cohen-Macaulay double points of a general projection into $\PP^4$; hence the name. 
Therefore, smooth surfaces in $\Pn^4$ are all special. 

Motivated by a conjecture of Hartshorne and Lichtenbaum  
about smooth rational surfaces, Ellingsrud and Peskine proved in 1989 the following:

\begin{theorem}[\cite{EP89}]
The degree of a smooth projective surface in $\PP^4$ of Kodaira dimension $\kappa(X)<2$ is bounded. 
Equivalently, only finitely many components of the Hilbert scheme of surfaces in $\PP^4$ contain points corresponding to smooth surfaces that are not of general type.    
\end{theorem}

At that time and earlier, there was a flourishing of activities to construct and classify smooth surfaces in $\PP^4$. Some of the constructions used Computer Algebra in an essential way.

In this paper, we give an overview of the ideas used to construct  examples of Kodaira dimension $<2$ and summarize the about 80 known families, of which about 40\% correspond to rational surfaces.

This might seem like a lot to report; however, from the point of view of ideas, there are not so many. We will devote one of the sections 
3 to 7 below to each new idea. 
Our accompanying Macaulay2 package ``NongeneralTypeSurfacesInP4'' \cite{ARS} is rather long. It contains code that produces a random surface over finite prime fields for each component known to us. These constructions prove the unirationality of most of the components.

Of course, since then the code and its speed have improved tremendously. In reviewing the constructions we discovered all together about 20 new families.
 
\subsubsection*{Open questions}
The following basic questions have motivated the search for new constructions of surfaces in $\Pn^4$. We supplement them with an extended  list of open problems in Section \ref{problems}.
\begin{itemize}
\item What is the maximum degree of a smooth rational surface in $\Pn^4$?
\item Is the irregularity $q=h^1(\sO_X)$ of a smooth surface in $\PP^4$
bounded above by $2$?
\item Is there a maximum degree for smooth surfaces in $\PP^4$ with a $(-1)$-line or even any $(-1)$-curve?
\end{itemize}

Ingrid Bauer, in \cite{Bau95}, showed that if $X$ is an inner projection from a smooth point of a surface in $\Pn^5$, then the degree is at most $9$.

\section{Birational classification}

Let $X$ be a smooth projective variety, and let $K_X$ denote a canonical divisor on $X$.
The canonical ring
\[R(X) =\oplus_{n\ge 0} H^0(X,nK_X)\]
is finitely generated by [BCHM10], and
\[\kappa(X)=\dim \Proj R(X)\]
is called the Kodaira dimension of $X$. If $\kappa(X) =\dim X$, then $X$ is called a variety of general type.

An irreducible curve $E$ on a smooth projective surface $X$ is called \textbf{(-1)-curve} if 
\[E^2= E.K_X=-1.\]
The adjunction formula implies $E \cong \PP^1$ and by Castelnuovo's contraction criterion \cite[Theorem V.5.7]{Har77} there exist a smooth surface $X'$ and a morphism $\sigma\colon X \to X'$ which contracts $E$
to a point $p \in X'$ and is biregular outside $E$. The morphism $\sigma$ coincides with the blowup of $X'$ at $p$. Both $X$ and $X'$ have the same canonical ring, the same irregularity $q=h^1(\sO_X)$, and the same geometric genus $p_g=h^0(\sO(K_X))=h^2(\sO_X)$. \medskip

\noindent
Since the topologies of $X'$ and $X$ differ by replacing $p$ with a $\PP^1$, we have $h^2(X',\QQ)=h^2(X,\QQ)-1$. Thus, one can blow-down (-1)-curves only finitely many times. 
A smooth projective surface is called \textbf{minimal} if it contains no (-1)-curves. 
Every smooth projective surface arises from a minimal surface by successively blowing up finitely many points. \medskip

\noindent
{\bf The Enriques-Kodaira classification.}
The Enriques-Kodaira classification groups minimal surfaces that are not of general type into several classes.

\begin{theorem}[e.g. \cites{BHPV04,Beau96}]
Let $X$ be a minimal surface of Kodaira dimension~$\kappa$ less than $2$.\\ 
   If $\kappa =-\infty$ then $X$ is isomorphic to $\PP^2$ or to a $\PP^1$-bundle over a curve $B$.\\
    If $\kappa=0$ then $X$ is one of the following 
    \begin{itemize}
        \item a K3 surface, $K_X=0$, $p_g=1$, $q=0$,
        \item an Enriques surface, $2K_X=0$, $p_g=0$, $q=0$,
        \item a bielliptic surface, $p_g=0$, $q=1$, such that $X \to E$ is a fiber bundle over an elliptic curve $E$ with fibers isomorphic to a fixed elliptic curve $F$,
        \item an abelian surface, $p_g=1$, $q=2$.
    \end{itemize}
  If $\kappa=1$, then $X$ is an elliptic fibration $X\to B$ over a curve $B$.
\end{theorem}

\noindent
{\bf Remarks.}
 \begin{enumerate}
     \item The minimal rational surfaces are $\PP^2$ and the Hirzebruch surfaces
     $F_n=\PP(\sO_{\PP_1}\oplus \sO_{\PP_1}(n)) \to \PP^1$ for $n\ge 0, \, n\not=1$.
     Note that $F_0 \cong \PP^1\times \PP^1$ and $F_1$ is isomorphic to the blow-up of $\PP^2$ at a point.
     \item Any two K3 surfaces are diffeomorphic as real 4-manifolds. Examples are
     quartic surfaces in $\PP^3$. The algebraic K3 surfaces form 19 dimensional
     subfamilies in a twenty dimensional family of possibly non-algebraic surfaces.
     \item Enriques surfaces are quotients of K3 surfaces by a fixed-point-free involution. They form 10-dimensional families.
     \item Bielliptic surfaces $X$ with the elliptic fibration
     $X \to E$, arise as finite quotients of a product of two elliptic curves $E'\times F$ by an automorphism $\tau$ which acts on $F$ with fixed points and on $E'$ by a
     translation. $E\cong E'/\tau$.
     \item Abelian surface arise as $\CC^2/L$ where $L \cong \ZZ^4$ is a lattice. For general $L$, these surfaces are non-algebraic.
     \item Some K3 surfaces, Enriques surfaces, and rational surfaces can also have elliptic fibrations, for example, the blowup $X$ of $\PP^2$ at the nine base points of a pencil of cubics. This example is also a surface that might have infinitely many (-1)-curves. Every section
     of the morphism $X \to \PP^1$ defined by the pencil is a (-1)-curve.
     \item The minimal model of a surface of Kodaira dimension $\kappa \ge 0$
     is uniquely determined.
     \item Surfaces of Kodaira dimension $\kappa=-\infty$ have many different
     minimal models. For example if $X\to B$ is a $\PP^1$-bundle and
     $p \in X$ a point, then the blowup $Bl_p(X)$ of $X$ at $p$ has another minimal model:
     Blowing down the strict transform of the fiber containing $p$ yields another minimal model.
 \end{enumerate}

\noindent
{\bf The adjoint linear system.}
The smooth surfaces not of general type in $\PP^4$ are usually not minimal. To compute their minimal models we can apply the adjunction process of Sommese and Van de Ven.

By adjunction, the adjoint divisor $H+K_X$ restricts to the canonical divisor $K_H$ on any smooth hyperplane section $H$ of $X$. So when $X$ is regular, the restriction of the map defined by the \textbf{adjoint linear system} $|H+K_X|$ to any smooth curve $H$ is the canonical map on $H$. 
The canonical map on a curve of genus $g\geq 2$ is an embedding if the curve is nonhyperelliptic, and a double cover of a rational normal curve if the curve is hyperelliptic. Thus, one expects a similar behavior for the map defined by $|H+K_X|$ on $X$. A (-1)-curve $E$ which is embedded as a line gets contracted
because $(H+K_X).E=0$. We call such a curve $E$ a (-1)-line.  Sommese and Van de Ven showed:

\begin{theorem}[\cite{SVdV87}, Adjunction]\label{adjunction mapping}  Let $X\subset \Pn^N$ be a smooth projective surface. The linear system $|H+K_X|$ defines a birational morphism
\[\varphi_{|H+K_X|}\colon X \to X' \subset \PP^{n'}\]
onto its image $X'$,
which blows down all (-1)-lines and is biregular otherwise, unless
\begin{enumerate}
    \item $X$ is a $\PP^2$ linearly or quadratically embedded or
    $X\to B$ is ruled by lines in $\PP^n$, in which case $|H+K_X|=\emptyset$.
    \item $X$ is a anti-canonical embedded Del Pezzo surface, in which case
    $\varphi_{|H+K_X|}$ maps $X$ to a point.
    \item $X \subset \PP^n$ is ruled by conics, in which case
    \[\varphi_{|H+K_X|}\colon X \to B \subset \PP^{n'}\]
    is the conic fibration.
    \item $X$ is a member of four families of surfaces, determined by Sommese and Van de Ven, in which case $\varphi_{|H+K_X|}$ defines a morphism that is generically finite to one onto its image instead of being birational.
\end{enumerate}
\end{theorem}

We introduce notation for a linear system on a non-minimal surface $X$ with minimal model $X_0$. If $\pi: X\to X_0$ is the blow-up of $k$ points $p_1, \dots, p_k$ and if $E_1,...,E_k$ are the corresponding exceptional divisors on $X$, then any complete linear system on $X$ without base components is of the form 
\[
|H|=|\pi^*D-\sum_{i=1}^ka_iE_i|
\]
where $D$ is a divisor on $X_0$ and $a_1, \dots, a_k \ge 0$ are multiplicities. We may interpret $|H|$ as the sublinear system of $|D|$ of divisors which have multiplicity at least $a_i$ at $p_i$ for $i=1,\ldots,k$. We will denote the projective surface with the indicated hyperplane class $H$ by 
\[
X_0(D;a_1,..,a_k)
\]
in case $H$ is very ample.

When $X_0=\PP^2$, the Picard group $\Pic X \cong \ZZ^{k+1}$ has generators corresponding to the pullback of a general line $L$ and the exceptional curves $E_i$ with a diagonal intersection matrix with entries $1,-1,\ldots,-1$.
With \[\PP^2(a;a_1,\ldots,a_k)=\PP^2(a;b_1^{m_1},\ldots,b_r^{m_r})\subset \PP^n,\]
where $m_i=|\{j\mid a_j=b_i\}|$,
we have two notations for a projective surface isomorphic to a blown-up $\PP^2$ with hyperplane class as indicated. Hence $\{m_1,\ldots m_r\}$ is a partition of $k$.

The position of the points to be blown up is not specified in this notation.
We will see later that for rational surfaces described below, the points have to lie in a special position, see Examples \ref{fss}, \ref{second special surface}, and also Proposition \ref{codimOfSpecialCollectionOfPoints}.

\begin{remark}(Exceptional families of Sommese and Van de Ven) The four exceptional families in $(4)$ are
\begin{enumerate}
    \item[(a)] $\PP^2(6;2^7) \subset \PP^6$, 
    \item[(b)] $\PP^2(6;2^7,1)\subset \PP^5$,
    \item[(c)] $\PP^2(9;3^8)\subset \PP^6$,
    \item[(d)] $\PP(\sE) \subset \PP^5$ where $\sE$ is an indecomposable rank 2 vector bundle of degree 1 on an elliptic curve embedded by $H=3B$, where $B$ denotes a section with $B^2=1$.
\end{enumerate}
The target of the adjunction map in cases (a), (b), and (d) is $\PP^2$, and 
 $\varphi_{|H+K_X|}\colon X \to \PP^2$ is generically 2:1 in cases (a) and (b), and generically 3:1 in case (d). 
 The birational involution of $\Pn^2$ defined by the double cover in cases (a) and (b) is known as the Geiser involution. In case (c) the adjunction map
$\varphi_{|H+K_X|}\colon X \to Q\subset \PP^3$ is generically 2:1 onto a quadric surface cone~$Q$ in $\PP^3$. This double cover defines a birational involution of the plane known as the Bertini involution. For details on Geiser and Bertini involutions see \cite{Dol12}. 
\end{remark}

If $X$ is not one of the exceptions $(1),\ldots, (4)$ in Theorem \ref{adjunction mapping}, then the image of a (-1)-conic in $X \subset \PP^n$ becomes a (-1)-line in $X' \subset \PP^{n'}$. Thus, by repeatedly applying the adjunction mapping, we obtain a sequence
\[X \to X_1 \to X_{2} \to \cdots \to X_r\]
of morphisms that ends either with one of the exceptions in Theorem \ref{adjunction mapping} or with a minimal surface. 

For rational surfaces, we will always encounter one of the exceptions.  Computing the exceptional curves at each step
provides a way to describe $X$ as a blowup of $X_{r}$ and analyzing the surface $X_{r}$ will yield a description of $X$ as the blowup of a minimal surface.\\

\noindent{\bf 6-secant lines.}
The variety of lines in $\PP^4$ that intersect a given surface $X \subset \PP^4$ has codimension one. Thus, since the Grassmannian $\GG(2,5)$ has dimension six, we expect that $X$ has finitely many $6$-secant lines.
Le Barz, in \cite{Leb81}, gave a formula for this number. There is a polynomial function
\begin{equation}\label{LeBarz}
N_6 = p(d,\pi,\chi)
\end{equation}
depending on the degree, sectional genus $\pi$, and holomorphic Euler characteristic $\chi$ of $X$, and its value coincides with
the number of $6$-secant lines plus the number of $(-1)$-lines, provided that there are only finitely $6$-secants.

If the homogeneous ideal of $X$ is generated by quintics, then there are no $6$-secant lines, and the number of $(-1)$-lines is a function of $d,\pi$ and $\chi$.

In general, we can obtain information about $6$-secants and $(-1)$-lines by considering the residual scheme of $X$ in the scheme $X_5$ defined by forms of degree less than or equal to~$5$ in the homogeneous ideal of $X$.

\begin{example}[$d=8$, $\pi=6$, and $p_g=q=0$]
\label{first special surface} 
A smooth surface $X\subset \PP^4$ with these invariants satisfies
\[K_X^2=-7,\, H.K_X= 2,\hbox{ and } N_6=12.\]
Hence, it has the intersection matrix
$$
\begin{pmatrix}
    H^2 & H.K_X \cr
    H.K_X & K_X^2
\end{pmatrix}
= \begin{pmatrix}
    8 & 2 \cr
    2 & -7
\end{pmatrix}. 
$$       

The homogeneous ideal of the surface $X$, constructed below in Example \ref{fss} via liaison, is generated by forms of degree less than or equal to $5$. Hence, Le Barz's formula indicates that $X$ has twelve $(-1)$-lines. Furthermore, the first adjoint surface $X_1$
has the intersection matrix
$$
\begin{pmatrix}
    H_1^2 & H_1.K_{X_1} \cr
    H_1.K_{X_1} & K_{X_1}^2
\end{pmatrix}
= \begin{pmatrix}
    5 & -5 \cr
    -5 & 5
\end{pmatrix}. 
$$
Since $H_1^2=(H+K_X)^2=5$, one obtains $H_1.K_{X_1}=(H+K_X).K_X=-5$, and
$K_{X_1}^2=K_X^2+12=5$.

By the Hodge index theorem \cite[Theorem V.1.9]{Har77}, we deduce $H_1=-K_{X_1}$. Thus $X_1$ is a Del Pezzo surface $\PP^2(3;1^4) \subset \PP^4$ of degree $5$. Since $X$ is the blowup of $\PP^2$ in $-K_X^2+9=7+9=16$ points, one gets $X=\Pn^2(6;2^4,1^{12})$.
\end{example}

\section{Linear system and liaison}

Using the adjunction mapping provides a first approach to finding smooth rational surfaces in $\PP^4$.  

Aure and Lanteri showed that smooth scrolls in $\Pn^4$ have degree $3$ or $5$ (\cite{Au87},\cite{La80}). The Veronese surface has degree $4$ and is the only other surface in $\Pn^4$ with an empty adjoint linear system. The adjoint linear systems of all other surfaces define morphisms, and Theorem \ref{adjunction mapping} applies.

Let $X\subset \Pn^4$ be a smooth rational surface. Then its embedding is given by a linear system determined by a divisor and a finite set of points on a minimal model. The minimal model is naturally obtained via the adjunction described above, whereas the points may require more care.  We give some examples.

\begin{example}\label{Bordiga1} Suppose $X\subset \Pn^4$ is a smooth rational surface of degree $6$ and sectional genus $3$. Then $H\cdot (H+K_X)=4$ and $K_X^2=-1,$ so $(H+K_X)^2=1$. Furthermore, $h^0(\sO_X(H+K_X)) = \chi(\sO_X(H+K_X)) = H.(H+K_X)/2 +\chi (\sO_X) = 3$. Thus, it follows from Theorem~\ref{adjunction mapping} that the adjunction mapping $\varphi_{H+K_X}:X \to \Pn^2$ is a birational morphism. Since $K_{\Pn^2}^2=9,$ the surface $X$ is the blowup of $K_{\Pn^2}^2-K_X^2=10$ points, and it contains ten disjoint $(-1)$-lines. So $X=\PP^2(4;1^{10})$. 

Conversely, the linear system of plane quartic curves through ten generic points is four-dimensional and yields a smooth surface of degree~$6$ and sectional genus~$3$ in $\PP^4$. The surfaces obtained in this way are called Bordiga surfaces, named after Giovanni Bordiga \cite {Bo1887}, and will play an important role in Example \ref{Abo surfaces}. 
\end{example}

Let $X=\PP^2(a;a_1,..,a_r) \subset \PP^4$ be a rational surface. 

Then the dimension of the linear system $|H|$ is at least 
\[
\chi(\sO_X(H))=\binom{a+2}{2}-\sum_{i=1}^r\binom{a_i+1}{2}.
\]
We say that $H$ is non-special if $h^0(\sO_X(H))=\chi(\sO_X(H))$, which occurs when the $\binom{a_i+1}{2}$ conditions imposed by the $r$ points are independent in the $\binom{a+2}{2}$-dimensional space of forms of degree $a$.
The difference
\[s=h^0(\sO_X(H))-\chi(\sO_X(H))= h^1((\sO_X(H))=5-\binom{a+2}{2}+\sum_{i=1}^r\binom{a_i+1}{2}\]
is called the speciality of $X$.

If a rational surface in $\PP^4$ is non-special (i.e., $s=0$), then its degree is at most~$9$. Furthermore, for each degree, there is one irreducible family of non-special rational surfaces in $\PP^4$ that is a dense set in an irreducible Hilbert scheme component, and the classification of these families was completed by Alexander \cite{Al88}.

If a rational surface $X$ in $\PP^4$ has a higher degree, then the collection of points $p_1, \dots, p_r$ blown up by $X \rightarrow \PP^2$ are in a special position. The collection is a point in $(\Pn^2)^r$.  Here we ignore the possibility of points being infinitely near for simplicity of argument. That case is entirely similar. 

\begin{proposition}\label{codimOfSpecialCollectionOfPoints}
The special collections of points leading to smooth rational surfaces $X$ have codimension $\le 5s$ in $(\PP^2)^r$ where $s=h^1(\sO_X(H))$.
\end{proposition}

To prove this recall the following theorem.
\begin{theorem}[Theorem III.12.11 of \cite{Har77}]
\label{minimal complex}
Let $(A,{\mathfrak m},K)$ be a local Noetherian ring and let $\sG$ be a coherent sheaf on $\PP^n_A$. Let $\pi\colon \PP^n_A \to \Spec A$ denotes the projection. If $\sG$ is flat over $\Spec A$ then the complex 
$R\pi_{\ast}\sG \in D^b(A)$ is represented by a minimal complex
\[0 \to A^{h^0} \to A^{h^1} \to \cdots \to A^{h^n} \to 0\]
of free A-modules, which is unique up to isomorphism.
The
formation of this complex commutes with base change, in particular 
$h^i= \dim H^i(\sG\otimes K)$.  
\end{theorem}

{\it Proof} of Proposition \ref{codimOfSpecialCollectionOfPoints}.
 Let $U \subset (\PP^2)^r$ be the open set of $r$ distinct points and let
$\widetilde X \to U$ denote the family of blowups $\PP^2$ in these point, 
Let $\sG= \sO_{\widetilde X}(\widetilde H)$ be the universal family of line bundles with assigned multiplicities.
Let $A$ be the local ring of the point corresponding to our surface. By the Theorem \ref{minimal complex}, 
$ R\pi_{\ast}(\sO_{\widetilde X}(\widetilde H))$ is locally represented by a complex 
\[ 0 \to A^5 \to A^s \to 0.\]
Thus the subvariety of good collections is locally defined by the entries of the $s\times 5$ middle matrix. By the principal ideal theorem \cite[Theorem 10.2]{Eis95}, its codimension is at most $5s$.  \qed \medskip

On the subvarieties of good collections in $(\Pn^2)^r$ very little is known beyond this expected codimension.  A characterization of the special collections of points has been given in three cases with $s=1$  and one case with $s=2$ (see Examples \ref{fss} and \ref{second special surface} and \cite{AR92},\cite{Ra88}), and is an open question for a number of rational surfaces with $s\geq 2$ that are otherwise known to exist.

Before we discuss further examples, consider liaison (also known as linkage), which sometimes provides a quick way to construct new surfaces from known ones.
\begin{remark}[Liaison]\label{linkage}
  Two surfaces $X$ and $Y$ in $\Pn^4$ are said to be (geometrically) linked if their union $X\cup Y$ is the complete intersection of two hypersurfaces. If  the two hypersurfaces have degree $e$ and $f$ respectively, then 
\[\deg X+ \deg Y= e\cdot f\] and we say that $X$ and $Y$ are linked in a $(e,f)$ complete intersection.   There are the following exact sequences of sheaves associated to linkage, cf. \cite{PS74},
\begin{equation*}
0\to \sO_X(K_X)\to \sO_{X\cup Y}(e+f-5)\to\sO_{Y}(e+f-5)\to 0
\end{equation*}
and
\begin{equation*}
0\to \sO_X(K_X)\to \sO_X(e+f-5)\to\sO_{X\cap Y}(e+f-5)\to 0.
\end{equation*}
These sequences for linkage between curves in $\Pn^3$ yield the following relation between the sectional genera:
\begin{equation*}
\pi(X)-\pi(Y)=\frac{1}{2}(e+f-4)({\rm deg} X-{\rm deg} Y).
\end{equation*}   
\end{remark}

\begin{example}[Example \ref{first special surface} continued] \label{fss} Consider a smooth rational surface $X\subset \Pn^4$  of degree $8$ and sectional genus $6$. 

A simple way to construct this surface is to use liaison.
The ideal of the surface is generated by a cubic form and some quartic forms.  This means that in the complete intersection of the cubic form and a general quartic form in the ideal, the residual (linked) surface is a surface $Y$ of degree four.  In fact, $Y$ is a projected Veronese surface. This linkage can, of course, be reversed to construct the surface $X$ as linked to a Veronese surface in a cubic and a quartic.  Since the ideal of the Veronese surface is generated by cubic forms, a general choice of linkage yields a smooth surface of degree 8, sectional genus 6, and $p_g=q=0$. So by Example \ref{first special surface} $X=\PP^2(6;2^4,1^{12}) \subset \PP^4$ with $s=1$.

We explain how one can choose the $16$ points next. For a different description, see \cite{CH97}.\medskip
\begin{proposition}[After Ellingsrud and Peskine]
Let $X=\PP^2(6;2^4,1^{12})\subset \PP^4$, be the blow-up in sixteen distinct points $p_1,..,p_4, q_1,..,q_{12}\subset \Pn^2$.
Then there are a quartic curve and a quintic curve in $\PP^2$ that are smooth at the points $p_i$ and $q_j$ with shared tangent lines at the $p_i$. In particular, the complete intersection of these curves is the union of the points $q_i$ and four schemes of length two, each supported at a corresponding $p_i$.
Conversely, a general such choice of sixteen points yields a smooth surface.
\end{proposition}

\begin{proof} Since we are primarily interested in computing examples, we only sketch why such configuration of points is sufficient to obtain a divisor $H$ with $h^1(\sO_X(H))=1$ and hence
$h^0(\sO_X(H))=5$.   
Consider the strict transform $C$ on $X$ of the plane quartic curve through the sixteen points $q_i$ and $p_i$, and the short exact sequence of sheaves
\[
0\to \cO_X(H-C)\to \cO_X(H)\to \cO_C(H)\to 0.
\]
The divisor $H-C$ on $X$ is represented by the strict transforms of the conics in the plane that pass through $p_1,\ldots,p_4$.  Therefore $h^1(\cO_X(H-C)(=h^2(\cO_X(H-C))=0$, so from the long exact sequence of cohomology we deduce that $h^1(\cO_X(H))=h^1(\cO_C(H))$. Hence 
$h^1(\cO_X(H))=1$ if and only if $h^1(\cO_C(H))=1$. But $C$ is a plane quartic curve and $\cO_C(H)$ has degree four, so $h^1(\cO_C(H))=1$ if and only if $\cO_C(H)=\omega_C=\cO_C(L)$, the canonical bundle on $C$, where $L$ is the pullback to $X$ of a line in $\Pn^2$.  Therefore $X$ has speciality $s=1$ if and only if $\cO_C(H)=\omega_C$. 

On the other hand, the existence of the quintic curve $D$ proves that 
\[
\cO_C(D-(\sum 2p_i+\sum q_i))=\cO_C(5L-(\sum 2p_i+\sum q_i))=\cO_C,
\]
so 
\[
\cO_C(H)=\cO_C(6L-(\sum 2p_i+\sum q_i))=\cO_C(L)=\omega_C. 
\]

For the proof of the very ampleness of $|H|$ we refer the reader to \cite{CH97}. 
Computationally, it is straightforward to give examples of smooth surfaces with this choice of points in $ \PP^2$. 
\end{proof}

\end{example}
The following is a slightly more involved example. It is the case of rational surfaces of degree $10$ and genus $9$ with one $6$-secant line.  For a full account and proofs see \cite{Ra88}.
\begin{example}\label{second special surface}
Consider a rational surface $X$ of degree $10$ and genus $9$ with one $6$-secant line.  Then $N_6(X)=7$, so $X$ contains six $(-1)$-lines.  By adjunction one deduces that the surface image of the second iterated adjunction may be a Veronese surface in $\Pn^5$. 

If so, let $$\pi:X\to \Pn^2,$$ be the composition of the two first adjunction mappings.  This map is then the blow up of twelve points $p_1,...,p_{12}$ and  six points $q_1,...,q_6,$ and 
$X=\PP^2(8;2^{12},1^6)\subset\Pn^4$.

In this case  $|H|$ has speciality $s=2$, so what is the position of the points $p_i$ and $q_j$?   

\begin{proposition} Let  $$X=\PP^2(8;2^{12},1^6)\subset \Pn^4$$  be  a smooth surface, and let $p_1,\ldots ,p_{12}, q_1,\ldots ,q_{6}\subset \Pn^2$ be the points blown up to obtain $X$. 
Let $\pi_1:X_1\to\Pn^2$ be the blow-up of $\PP^2$ at the points $p_1,\ldots ,p_{12}$, and let $q_i, i=1,\ldots ,6$, denote their preimages on $X_1$; i.e., the surface $X$ is the blow-up of $X_1$ at these $q_i$.
Assume that the points $p_i$ are sufficiently general so that the linear system $|C|=|4\pi^*L-\sum_{i=1}^{12} E_i|$ on $X_1$ defines a $(4:1)$ map $$\rho: X_1\to \Pn^2.$$ 
Then there is a line $L_0\subset \Pn^2$ with preimage $L_1= \pi_1^{-1}L_0$  such that $\rho(L_1)$ is a plane quartic with three nodes. The six points $q_i$ together with six points on $L_1$ form the twelve points on $X_1$ that are mapped by $\rho$ to the three nodes of $\rho(L_1)$.

Conversely, a general such choice of twelve points $p_i$ and a line $L_0$  determines six points $q_i$ that yield a smooth surface.  
\end{proposition}

\begin{proof} 
In this proof, we only address how the choice of these points leads to the conditions $h^1(\sO_X(H))=2$ and $h^0(\sO_X(H))=5$. 
The key to the choice of points is that the strict transforms on $X$ of the preimages under $\rho$ of the three lines through two nodes of $\rho(L_1)$ are embedded as canonical curves, i.e., as plane quartic curves $C_1, C_2, C_3$ on $X\subset \PP^4$. The three planes of these curves share a line, which is
the unique $6$-secant line to $X$. The six points in $X$ on this line in $\Pn^4$ all lie in the preimage of the line $L_0\subset \Pn^2$ under $\pi$ and map in pairs to the nodes of $\rho(L_1)$.
The speciality $s=2$ of $|H|$, i.e. that $h^0(\cO_{X}(H))=5$, follows from the existence  of the special line $L_0$ and these three plane quartic curves. 
We now give a more detailed argument for these facts.  

The linear system $|H|$ on $X$ and the linear subsystem on $X_1$ of curves in $|2C|$ passing through the points $q_1,\ldots, q_6$ are in one-to-one correspondence.
Thus, we need to show that this linear system is $4$-dimensional. Since $h^0(\cO_{X_1}(2C))=9$, it is equivalent to show that the subspace $U\subset H^0(\cO_{X_1}(2C))$ of sections vanishing at $q_1,\ldots,q_6$ has codimension $4$. 

Consider $\overline{C}_i, i=1,2,3$, the preimage under $\rho$ in $X_1$ of the lines through two of the three nodes of $\rho(L_1)$. It is the isomorphic image of $C_i\subset X$ in $X_1$.  Let $q'_3+q'_4$ and $q'_5+q'_6$  be the preimages in $L_1$ of two of the nodes. 
Then $q_3+q_4+q'_3+q'_4$ and $q_5+q_6+q'_5+q'_6$ are the preimages in $\overline{C}_1$ of the two nodes, and   $L_1\cap\overline{C}_1=q'_3+q'_4+q'_5+q'_6$. Similarly, $q_3+q_4+q'_3+q'_4$ and $q_1+q_2+q'_1+q'_2$ are the preimages in $\overline{C}_2$, and $q_5+q_6+q'_5+q'_6$ and $q_1+q_2+q'_1+q'_2$ are the preimages in $\overline{C}_3$ of a pair of nodes on $\rho(L_1)$.  So, 
\[\overline{C}_1\cap \overline{C}_2=q_3+q_4+q'_3+q'_4\quad {\rm and}\quad\overline{C}_1\cap \overline{C}_3=q_5+q_6+q'_5+q'_6. \]
By translating 
from $C_1$ to $\overline{C}_1$, we have:
\[\omega_{\overline{C}_1}=\cO_{\overline{C}_1}(L_1)=\cO_{\overline{C}_1}(q'_3+q'_4+q'_5+q'_6)=\cO_{C_1}(q'_3+q'_4+q'_5+q'_6)=\omega_{C_1}.\]
Furthermore, 
\[\cO_{\overline{C}_1}(2C-(q_3+q_4+q_5+q_6))=\cO_{\overline{C}_1}(\overline{C}_2+\overline{C}_3-(q_3+q_4+q_5+q_6))=\cO_{\overline{C}_1}(q'_3+q'_4+q'_5+q'_6),
\]
where the first equality follows by adjunction, so we obtain
\[
h^0(\cO_{\overline{C}_1}(2C-(q_3+q_4+q_5+q_6)))=h^0(\omega_{\overline{C}_1})=3.
\]
In particular, $H^0(\cO_{\overline{C}_1}(2C-(q_3+q_4+q_5+q_6)))$ has codimension $3$ in $H^0(\cO_{\overline{C}_1}(2C)).$

Consider the exact sequence of sheaves on $X_1$
\[
0\to\cO_{X_1}(C)\to \cO_{X_1}(2C)\to \cO_{\overline{C}_1}(2C)\to 0.
\]
Notice that $h^0(\cO_{X_1}(2C))=9$ and $h^0(\cO_{\overline{C}_1}(2C))=6.$
By assumption, $h^0(\cO_{X_1}(C))=3$,  so  $h^1(\cO_{X_1}(C))=0.$ Therefore, the map in cohomology $$\sigma:H^0(\cO_{X_1}(2C))\to H^0(\cO_{\overline{C}_1}(2C))$$ is surjective.  
Since the subspace $H^0(\cO_{\overline{C}_1}(2C-(q_3+q_4+q_5+q_6)))$ of $H^0(\cO_{\overline{C}_1}(2C))$ has  codimension $3$,  
so does the linear subspace $U_1$ of curves in $|2C|$ passing through the four points $q_3+q_4+q_5+q_6$, because the underling vector space of $U_1$ is the preimage of $H^0(\cO_{\overline{C}_1}(2C-(q_3+q_4+q_5+q_6)))$ under $\sigma$. 

By a similar argument for $C_2$ and $C_3$, we conclude that the linear subspaces $U_2$ and $U_3$ of curves in $|2C|$ passing through $q_1+q_2+q_5+q_6$ and $q_1+q_2+q_3+q_4$ respectively have codimension $3$.   

The linear subspace of curves in $|2C|$ passing through the two points $q_1+q_2$ has codimension $2$, and the same applies for $q_3+q_4$ and $q_5+q_5$. Therefore, the span of $U_2$ and $U_3$ has codimension at least $2$ in $|2C|$. Similarly, the linear space spanned by $U_1$ and $U_3$ and the span of $U_1$ and $U_2$ also have codimension at least $2$ in $|2C|$. This is possible only if the span of all three $U_1$, $U_2$, and $U_3$ has codimension $2$; this means either that it coincides with the span of any two of them or that $U_1\cap U_2\cap U_3$ has codimension $4$. However, the former would mean that any curve in each $U_i$ contains all six points $q_i$, and hence the $U_i$ coincide, which is absurd. Thus, $U_1\cap U_2\cap U_3$, the linear subspace of curves in $|2C|$ passing through $q_1+\ldots +q_6$, has codimension $4$ in $|2C|$, as claimed. 

Finally, let us justify that the curves $C_i$ are plane quartic curves on $X$, which implies that $\cO_{C_i}(H)=\omega_{C_i}$ for each $i$.

Recall that 
$\cO_{\overline{C}_1}(2C-(q_3+q_4+q_5+q_6))=\omega_{\overline{C}_1}$, from which it follows that $\cO_{C_1}(H)=\omega_{C_1}$.
Consider the exact sequence of sheaves on $X$
\[
0\to\cO_{X}(H-C_1)\to \cO_{X}(H)\to \cO_{C_1}(H)\to 0.
\]
Since $h^0(\cO_X(H))=5$ and since $|H-C_1|$ and the pencil of curves in $|C|$ passing through $q_1+q_2+q'_1+q'_2$ on $X_1$ are in one-to-one correspondence, we have $h^0(\cO_X(H-C_1)=2$, and hence the map in cohomology 
$H^0(\cO_X(H))\to H^0(\cO_{C_1}(H))$ is surjective.  Therefore, $C_1$, and similarly $C_2$ and $C_3$, are plane quartic curves on $X\subset\Pn^4$.

For necessity, a key point is to observe that the preimages of the six points on the 6-secant line in $\Pn^4$ lie on a line in $\Pn^2$. This gives the desired line $L_0$.
\end{proof}
Similarly to Example~\ref{fss}, the smooth surfaces $X \subset \PP^4$ from Example \ref{second special surface} allow for a construction with liaison. They are linked to reducible surfaces: The surface $X$ is contained in a pencil of quartic hypersurfaces, cf. \cite[Proposition 2.9]{Ra88}. By Bezout's theorem, the $6$-secant line $\ell$ and the planes of the three plane quartic curves $C_i,\, i=1,2,3,$ that also contain $\ell$ are contained in these hypersurfaces. The surface linked to $X$ in a $(4,4)$ complete intersection is therefore reducible: it contains the three planes $P_i$ of the curves $C_i$ and an additional fourth component, a surface $Y$ of degree $3$ that turns out to be a rational cubic scroll whose directrix is $\ell$. The surface $Y$ intersects each of the three planes $P_i$ scheme-theoretically only along $\ell$. It is straightforward to construct the union of three planes and a rational scroll $Y$ that satisfies this property. Thus, the surface $X$ appears via linkage in a general complete intersection of two quartics containing $Y$.
\end{example}

In our accompanying Macaulay2 package we actually neither use linear 
systems nor liaison for Examples \ref{fss} and \ref{second special surface}. Instead, we use a construction
with Tate resolution introduced in Section \ref{tate}.

As these examples indicate, determining the collection of base points for a linear system of curves on $\Pn^2$, that defines a rational map with a smooth image $X\subset\Pn^4$, requires a  case-by-case analysis as the degree exceeds $9$.  
  
\begin{remark}
If $X=\PP^2(a;a_1,..,a_k)\subset \Pn^4$ is not any of the exceptions in Theorem~\ref{adjunction mapping},  with the $a_i$ in decreasing order, then the adjoint linear system $|H+K|$ defines a smooth surface $X_1=\PP^2(a-3;a_1-1,..,a_k-1)\subset \Pn^{N_1}$, with $N_1=\pi-1$, that is nonspecial: it restricts to the complete canonical linear system on $H$, and hence by the Riemann-Roch theorem, $\chi({\mathcal O}_X(H+K))=\pi=h^0({\mathcal O}_X(H+K))$. So,  if $r$ $(-1)$-lines in $X$ are blown down by the adjunction mapping, then the $k-r$ points $p_1,\dots,p_{k-r}$ in $\Pn^2$ with multiplicities $a_i-1>0$ of the adjoint linear system impose independent conditions on curves of degree $a-3$, as is expected for a general choice of $k-r$ points. By Proposition \ref{codimOfSpecialCollectionOfPoints}, one expects  the codimension in the Hilbert scheme of points in $\Pn^2$ of the collection of $k$ points such that a linear system of type $\PP^2(a;a_1,..,a_k)$ has speciality $s$ is $5s$. So as soon as $2r<5s$, the points $p_1,\dots,p_{k-r}$ are in special position relative to $|H|$, even though, with multiplicities $a_i-1$, they impose independent conditions on curves of degree $a-3$.
\end{remark}

Given the degree $d$ and the sectional genus of a smooth rational surface $X$ in $\Pn^4$, one may iteratively apply the adjunction mappings to identify possible numerical types $\PP^2(a;a_1,..,a_k)$ for $X$, and then look for conditions on the $k$ points that are blown up to obtain a four-dimensional very ample linear system.

For degree $d=10$ and the possible sectional genera, only three of $23$ numerically possible linear systems are realized on smooth surfaces, c.f. \cite{Ra88}. Therefore, one was looking for different methods of constructions.

   \section{Hilbert-Burch}\label{Section: Hilbert Burch} 

An alternative approach to constructing smooth surfaces in $\Pn^4$ involves considering determinantal equations. This method starts with equations that set the minors of a matrix to zero, extended to the construction of a surface in $\Pn^4$ as the degeneracy locus of a map
\[
\phi\colon \sF\to \sG
\]
between vector bundles $\sF$ and $\sG$ of ranks $f$ and $f+1$, respectively. 

If the dependency locus of $\phi$ is a surface, then the Eagon-Nothcott  complex \cite[Thm A2.10]{Eis95} of $\phi$
\[0 \to \sF \to \sG \to \sO_{\PP^4}(c_1(\sG)-c_1(\sF))\to \sO_X(c_1(\sG)-c_1(\sF))\to 0\]
is exact. After twisting by $\sO_{\PP^4}(-c_1(\sG)+c_1(\sF))$, we may assume that
$c_1(\sF)=c_1(\sG),$ so that the above sequence induces a locally free resolution 
\begin{equation}\label{eq:lfr} 0 \to \sF \to \sG \to \sI_X\to 0\end{equation}
of the ideal sheaf of $X$, from which one can derive the invariants of $X$.

The simplest case is when $\sF$ and $\sG$ are direct sums of line bundles. The space of maps from $\sF$ to $\sG$ can be identified with the space of matrices with homogeneous entries, and a general map defines a smooth surface if the degrees of all the entries are positive. See \cite{Ell75,EP91} for a sharp statement. 

\begin{example}[Example \ref{Bordiga1} continued]\label{Bordiga} A Bordiga surface $X\subset \Pn^4$ is the rank two locus of a map 
\[\mathcal{O}^3_{\Pn^4}(-4)\to \mathcal{O}^4_{\Pn^4}(-3),\]
i.e., the rank two locus of a $3\times 4$-matrix $M_X$ of linear forms. The relation among the rows at a rank two point yields a point in $\Pn^2$, defining a map $\phi_X:X\to \Pn^2$.

The matrix $M_X$, viewed as a $3\times 4\times 5$ tensor, has a $3\times 5$ adjoint matrix $L_X$ with linear entries in $4$ variables and a $4\times 5$ adjoint matrix $N_X$ with linear entries in $3$ variables.  The matrix $L_X$ has rank $2$ at ten points in $\Pn^3$ that correspond to ten planes in $\Pn^4$ intersecting the Bordiga surface in a plane cubic curve.  These planes are defined by the generalized columns of $M_X$ with dependent entries. The matrix $N_X$ vanishes at the ten points in the plane, the images of ten exceptional lines in the Bordiga surface contracted to a point by $\phi_X$.  The rational parametrization $\PP^2 \dashrightarrow X \subset \PP^4$ is given by the five $4\times 4$-minors of $N_X$.
\end{example}

The vector bundles $\sF$ and $\sG$ are direct sums of line bundles
if and only if $X$ is arithmetically Cohen-Macaulay (aCM). However, in most cases,  the surface $X$ is not aCM. Consequently, if $X$ is non-aCM and obtained via a locally free resolution~(\ref{eq:lfr}), then $\sF$ or $
\sG$ (or both) is not a direct sum of line bundles, indicating at least one of them has nonzero intermediate cohomology groups by Horrocks' splitting criterion~\cite{Horrocks64}.

If $X$ is non-aCM, then some twist of its ideal sheaf has a nonzero first or second cohomology group. The degree to which $X$ is not aCM can be quantified by the modules $H^i_* (\sI_X) = \oplus_{k \in \ZZ} H^i(\sI_X(k))$, $i \in \{1,2\}$, called the Hartshorne-Rao modules of $X$. These modules are linked to the intermediate cohomology modules $H^i_*(\sG) = \oplus_{k \in \ZZ} H^i(\sG(k))$ and $H^i_*(\sF) = \oplus_{k \in \ZZ} H^i(\sF(k))$ of $\sG$ and $\sF$, respectively, if $\sI_X$ admits a locally free resolution~(\ref{eq:lfr}). We aim to explore a method to identify vector bundles $\sF$ and $\sG$ with $H^1_*(\sI) = H^1_*(\sG)$ and $H^2_*(\sI_X) = H^3_*(\sF)$ such that $X$ appears as the dependency locus of a map from $\sF$ to $\sG$.

Consider the standard short exact sequence for $\sI_X$.  
\[0 \to \sI_X \to \sO_{\PP^4} \to \sO_X \to 0.\]
From this sequence, one can derive the polynomial
$\chi(\sI_X(k))$ from the basic numerical invariants of our surface:
\[\chi(\sI_X(k)) = \binom{k+4}{4}-\chi(\sO_X(kH)) \in \QQ[k]
\]
is a polynomial of degree four with coefficients depending only on the basic numerical invariants of our surface. In fact, 
by Riemann-Roch,
\[\chi(\sO_X(kH))= \frac{1}{2}kH\cdot(kH-K_X) +\chi(\sO_X)=\binom{k+1}{2}d- k\pi+k+\chi(\sO_X), \]
where $d=\deg  X$ is the degree of $X$, $\pi=\frac{1}{2}H\cdot(H+K_X)+1$ denotes the sectional genus, and $\chi(\sO_X)=1-q+p_g.$%

If the ideal sheaf $\sI_X$ of $X$ has natural cohomology, i.e., $h^i(\sI(k))\not=0$ for at most one $i\in \{0,\ldots,4\}$ for all $k\ge -4$, then 
\[h^i(\sI_X(k))=(-1)^i \chi(\sI_X(k)),  \]
allowing us to determine a plausible cohomology table for $\sI_X$, i.e., the table with rows indicating the cohomology degree $i$, diagonals corresponding to the twists $k$ of $\sI_X$, and cells reflecting the dimension of the cohomology $H^i(\sI_X(k-i))$. 
\begin{example}[$d=10$, $\pi=8$, and $p_g=q=0$]
Consider a surface $X \subset \PP^4$ of degree $d=10$, sectional genus $\pi=8$, and  $q=p_g=0$, so $\chi(\sO_X)=1$.
A plausible cohomology table 
is
$$\begin{array}{c|ccccc ccccc}
        & -1 & 0 & 1 & 2 & 3 & 4 & 5 & 6 & 7  \\ \hline 
     4 & 1 & . & . & . & . & . & . & . & . \\
     3 & 89 & 52 & 25 & 8 & . & . & . & . & . \\
     2 & . & . & . & . & 1 & . & . & . & . \\
     1 & . & . & . & . & 2 & 5 & 3 & . & . \\
     0 & . & . & . & . & . & . & 10 & 41 & 98 \\
     \end{array}$$
Note that the entry 1 in the $(4,-1)$ cell comes from $h^4(\sO_{\PP^4}(-5))=1$.
\end{example}

The proposition below shows how to obtain vector bundles $\sF$ and $\sG$ such that $H^1_*(\sI_X)=H^1_*(\sG)$ and $H^2_*(\sI_X)=H^3_*(\sF)$.

\begin{proposition}[Syzygy bundles, e.g.\cite{DES93}] Let $M$ be a finite length graded module over the homogeneous coordinate ring $S=k[x_0,\ldots,x_n]$ of $\PP^n$.
Consider its minimal free resolution:
$$0 \leftarrow M \leftarrow F_0 \leftarrow F_1 \leftarrow \ldots \leftarrow 
F_{n+1} \leftarrow 0.$$
 Let $\sF_i=\syz_i M$ for $1\le i\le n-1$ denote the sheafification of the $i$-th syzygy module $\ker(F_i \to F_{i-1})$ of $M$. Then
the $i$-th cohomology module is $H^i_*(\sF_i)\cong M$ and all other intermediate cohomology of $\sF_i$ vanishes. Conversely,
if $\sE$ is a vector bundle on $\PP^n$ with intermediate cohomology $H^i_*(\sE)\cong M$ and
$H^j_*(\sE)=0$ for $j\not=i$, then $\sE \cong \sF_i \oplus \sL$ where $\sL$ is a direct sum of line bundles. 
\end{proposition}

The first examples of syzygy bundles are the exterior powers $\Omega^i(i)$ of the cotangent bundle on $\Pn^n$ twisted by $i$.
\begin{example} From the tautological short exact sequence
\[0 \to U \to \sO_{\PP^n}^{n+1} \to \sO_{\PP^n}(1) \to 0\]
we see that $U \cong \syz_1 K(1)$ is the sheafification of the syzygy module
of the residue module $K=S/(x_0,\ldots,x_n)$ twisted by $1$. We also see
$U\cong \Omega^1(1)$ by comparing this sequence with the Euler sequence. More generally we have \[\Lambda^i U \cong \Omega^i(i) \cong \syz_i K(i).\]
\end{example}. 

\medskip
\noindent
{\bf Example 4.2} continued. We are led to  
take $\sG$ as the first syzygy sheaf of a module $M$ with Hilbert function $\{2,5,3,0,\ldots\}$ and  to take $\sF$ as the $3$-rd syzygy sheaf of
$K(-1)$.    

The module $M$ has a presentation 
\[ 0 \leftarrow M \leftarrow S(-2)^2\leftarrow S(-3)^5\oplus S(-4)^2 \]
defined by the concatenation of a $2\times 5$ matrix $m_{2\times 5}$ with linear entries and a $2\times 2$ matrix $m_{2\times 2}$ with quadratic entries. 
Since $\rank \sG=5$ and $\rank \sF=4$ we may not need line bundle summands.

The minors of a general $2\times 5$ matrix with linear entries vanish in $5$ points. After change of coordinates of $\PP^4$ we assume that these are the coordinate points of $\PP^4$ and that 
\[m_{2\times 5}= \begin{pmatrix}
    x_0 & x_1 & x_2 &x_3 & x_4 \\
    a_0x_0 & a_1x_1 &a_2x_2 &a_3x_4 & a_4x_4\\
\end{pmatrix}\]
for five distinct scalors $a_i\in K$. By applying an automorphism of $S(-3)^5\oplus S(-4)^2$, we may assume that
\[ m_{2\times 2} = \begin{pmatrix}
    0 & 0 \\
    q_1 & q_2
\end{pmatrix}\]
for quadrics $q_1, q_2$. These quadrics can be chosen to lie in the span of the squares $x_0^2,\ldots,x_4^2$, because the monomials
$x_ix_j$ for $i\not=j$ are in the annihilator of $\coker( m_{2\times 5})$.
The betti table of $\coker(m_{2\times 5}|m_{2\times 2})$ for a general presentation matrix is
 \[\begin{array}{rcccccc}
         & 0 & 1 & 2 & 3 & 4 & 5\\
      \text{total:} & 2 & 7 & 15 & 20 & 13 & 3\\
      2: & 2 & 5 & . & . & . & .\\
      3: & . & 2 & 15 & 10 & 3 & .\\
      4: & . & . & . & 10 & 10 & 3
      \end{array}\]
The Buchsbaum-Rim complex \cite[Thm A2.10]{Eis95} associated to $m_{2\times 5}$  is a subcomplex with betti table
\[ \begin{array}{rcccccc}
       & 0 & 1 & 2 & 3 & 4\\
      \text{total:} & 2 & 5 & 10 & 10 & 3\\
      2: & 2 & 5 & . & . & .\\
      3: & . & . & 10 & 10 & 3
      \end{array}\]
The vector bundle $\sF$ has a resolution
\[0 \leftarrow \sF \leftarrow \sO^{5}(-5) \leftarrow \sO(-6)  \leftarrow 0. \]
Compared with resolution of the vector bundle $\sG$ for a general choice of $m_{2\times 5}|m_{2\times 2}$
\[0 \leftarrow \sG \leftarrow \sO^{15}(-5) \leftarrow \sO^{10}(-6)\oplus \sO^{10}(-7) \leftarrow \sO^3(-7)\oplus \sO^{10}(-8) \leftarrow \sO^3(-9)\leftarrow 0,\]
we see that the image of a morphism $\varphi \in \Hom(\sF,\sG)$ is contained in the sheafified kernel $\ker(m_{2\times 5}\colon S^5(-3)\to S^2(-2))$.
Since this sheaf has only rank $3$, the homomorphism~$\varphi$ has rank $\le 3$ everywhere, so we do not obtain a surface. We have to do better; we need to find a $m_{2 \times 2}$ such that the cokernel of $(m_{2 \times 5} | m_{2 \times 2})$ has additional third syzygies in degree $6$.

In the diagram of a morphism $\varphi$, 
$$\xymatrix{ 
0 &\ar[l] \sG & \ar[l] \sO^{15}(-5)\oplus \sO^{k}(-6)&\ar[l] \sO^{10+k}(-6)\oplus \sO^{10}(-7)   \\
0 &\ar[l] \ar[u]^\varphi \sF & \ar[l]\ar[u] \sO^{5}(-5) & \ar[l] \ar[u] \sO(-6) 
}$$
we see that the vertical map from $\sO(-6)$ can only hit second syzygies of $\ker(m_{2\times 5})$ if $k=0$. To get $k>0$, consider syzygies of $\sG$ arising from $m_{2\times 2}$.

The additional five first syzygies of $(m_{2\times 5}|m_{2\times 2})$ can be explained as follows:
In the pencil of quadrics $sq_1+tq_2$, counted with multiplicities, there are precisely five quadrics that depend on at most four of the five squares $x_0^2,\ldots,x_4^2$. If $x_i^2$ is the missing square in $sq_1+tq_2$, then $x_i(sq_1+tq_2)$ lies in the ideal $(x_ix_j \mid i\not=j)$ and the vector 
\[ x_i \begin{pmatrix}
0\\ sq_1+tq_2 \\ 
\end{pmatrix}\]
is a quadratic linear combination of the columns of $m_{2\times 5}$, and hence it gives rise to a syzygy of $(m_{2\times 5}|m_{2\times 2})$.

If we choose quadrics $q_1=\sum_{i=0}^2 b_i x_i^2$ and $q_2=\sum_{i=2}^4 c_i x_i^2$
with non-zero scalars $b_0,b_1,b_2,c_2,c_3,c_4 \in K$, then additional syzygies arise 
\[\begin{array}{rcccccc}
         & 0 & 1 & 2 & 3 & 4 & 5\\
      \text{total:} & 2 & 7 & 15 & 20 & 13 & 3\\
      2: & 2 & 5 & . & . & . & .\\
      3: & . & 2 & 15 & 12 & 3 & .\\
      4: & . & . & 2 & 10 & 10 & 3
      \end{array}\]
and a general $\varphi \in \Hom(\sF,\sG)$ yields a smooth surface $X$ of degree $d=10$, sectional genus $\pi=8$, and $p_g=q=0$.

It turns out that $X$ is an Enriques surface blown up at $4$ points, and thus $K_X^2=-4$. The adjunction process
\[ X \to X_1 \subset \PP^7\]
blows down four $(-1)$-lines, and the resulting surface $X_1$ is a minimal Enriques surface. 

If we set $b_2=0$, then we obtain even more syzygies
\[\begin{array}{rcccccc}
         & 0 & 1 & 2 & 3 & 4 & 5\\
      \text{total:} & 2 & 7 & 15 & 20 & 13 & 3\\
      2: & 2 & 5 & . & . & . & .\\
      3: & . & 2 & 15 & 14 & 4 & .\\
      4: & . & . & 4 & 11 & 10 & 3
      \end{array}\]
In this case, we get a rational surface $X$. The adjunction process 
\[ X \to X_1 \to X_2 \to X_3 \subset \PP^5\]
blows down two $(-1)$-lines and one $(-1)$-conic. The final surface $X_3 \subset \PP^5$
is a conic bundle over $\PP^1$ with nine singular fibers. Thus, the self-intersection number of the canonical divisor is $K^2_X=8-9-1-2=-4$, as before.
\qed

It is interesting that the construction of the Enriques surface differs by an $\epsilon=b_2$ from that of the rational surface. For an $\epsilon\to 0$ and 
 $\varphi$ in the limit space $\Hom(\sF,\sG_\epsilon) \subset \Hom(\sF,\sG_0)$, the resulting surface $X_0$ lies in the intersection of the two components of the Hilbert scheme. 
 
 \begin{question} What are the singularities of $X_0$, and how is it a limit of both non-minimal Enriques surfaces and rational surfaces?
 \end{question}
 
\bigskip

\section{Random searches over finite fields}

If a parameter space $\MM$ (e.g., a Hilbert scheme or a moduli space) for interesting objects in algebraic geometry has low codimension $c$ in a unirational
variety $\GG$, then over a finite field $\FF_q$ with $q$ elements, we expect that the ratio
\[\frac{|\MM(\FF_q)|}{|\GG(\FF_q)|} \approx \frac{1}{q^c}.\]
Thus, if $q$ and $c$ are small and we can {\bf check fast} whether a randomly chosen point $p\in \GG(\FF_q)$ is {\bf not in $\MM(\FF_q)$}, we might be able to find a point in $\MM(\FF_q)$ via a Computer Algebra search within a reasonable time frame.

To our knowledge, this approach to constructing an interesting variety was first introduced in \cite{Sch96} for the following example.

\begin{example}[$d=11$, $\pi=10$, and $p_g=q=0$]\label{schreyer surfaces}
Let $X \subset \PP^4$ be a smooth surface of degree $d=11$, sectional genus $\pi=10$, and $p_g=q=0$. The double point formula shows that $K^2=-6$.     

If $X$ has natural cohomology, then it has the following cohomology table 

$$\begin{array}{c|ccccc ccccc}
        & -1 & 0 & 1 & 2 & 3 & 4 & 5 & 6 & 7  \\ \hline 
     -4 & 1 & . & . & . & . & . & . & . & . \\
     -3 & 103 & 61 & 30 & 10 & . & . & . & . & . \\
     -2 & . & . & . & . & 2 & . & . & . & . \\
     -1 & . & . & . & . & 1 & 5 & 5 & . & . \\
     0 & . & . & . & . & . & . & 5 & 32 & 84 \\
     \end{array}$$
  
This table shows that $H^1_*(\sI_X)$ has Hilbert function $(1,5,5,0, \ldots)$. Modules with presentation
\[0 \leftarrow M \leftarrow S(-2) \leftarrow S^{10}(-4)\]
have this Hilbert function. They form an open subset of the space of modules with the desired Hilbert function and correspond to points in the Grassmannian $\GG=\GG(10,H^0(\PP^4,\sO(2)))$.
We would like to take $\sG=\syz_1 M$  and $\sF=\syz_3(2K(1))$. 

If $M \in \GG$ is generic, then $\sG$ has a resolution
\[0 \leftarrow \sG \leftarrow\sO^{15}(-5)\oplus\sO^5(-6) 
\leftarrow\sO^{26}(-7)\leftarrow \sO^{20}(-8)\leftarrow\sO^5(-9)\leftarrow 0.\]
Since 
\[0 \leftarrow \sF \leftarrow\sO^{10}(-5)\leftarrow \sO^2(-6) \leftarrow 0\]
and there are no nonzero morphisms from the second term in the locally free resolution of $\sF$ to the second term in the locally free resolution of $\sG$, we have $\Hom(\sF,\sG)=0$.
Thus, we need to construct a module $M$ with Betti table
\[\begin{array}{rcccccc}
         & 0 & 1 & 2 & 3 & 4 & 5\\
      \text{total:} & 1 & 10 & 22 & 28 & 20 & 5\\
      2: & 1 & . & . & . & . & .\\
      3: & . & 10 & 15 & k & . & .\\
      4: & . & . & 5+k & 26 & 20 & 5
      \end{array}\]
with $k \geq 2$, which could lead to a diagram
$$\xymatrix{ 
0 &\ar[l] \sG & \ar[l] \sO^{15}(-5)\oplus \sO^{5+k}(-6) & \ar[l] \sO^k(-6)\oplus \sO^{26}(-7) \\
0 &\ar[l] \ar[u]^\varphi \sF & \ar[l]\ar[u] \sO^{10}(-5) & \ar[l] \ar[u] \sO^2(-6)
}$$

Consider 
\[\MM= \{M \in \GG\mid \dim \Tor^S_3(M,K)_6 \ge 2\}.\]
We know very little about $\MM$. It is a subscheme of expected codimension $14=7\cdot 2$. So, if we work with $\FF_3$, then we expect to find a point in $\MM(\FF_3)$ by testing about $3^{14}\approx 4,700,000$ cases.
Fortunately we can improve our search by forcing one second syzygy.

Consider the Koszul complex in our five variables, along with ten quadrics
obtained by composing a vector $\alpha$ of ten linear forms with the $10\times 10$ Koszul matrix $kos_3$:

$$\xymatrix{
S & & & &\\
\ar[u]^\alpha S^{10}(-1) & \ar[l]^{kos_3} \ar[lu]_\beta S^{10}(-2) & \ar[l] S^5(-3) & \ar[l] S(-4) & \ar[l] 0 \\
}
$$
For a general choice of $\alpha$, we obtain a module with Hilbert function $(1,5,5,0,\ldots)$ and syzygies
\[\begin{array}{rcccccc}
         & 0 & 1 & 2 & 3 & 4 & 5\\
      \text{total:} & 1 & 10 & 22 & 28 & 20 & 5\\
      2: & 1 & . & . & . & . & .\\
      3: & . & 10 & 15 & 1 & . & .\\
      4: & . & . & 6 & 26 & 20 & 5
      \end{array}\]
The image $\GG_1$ of this family of modules in the $50$-dimensional Grassmannian $\GG=\GG(10,15)$ has the expected codimension $6$. This is because the choice of $\alpha$ and $\alpha'$ with $\alpha-\alpha'= \gamma\circ kos_2$ yields the same $\beta$, where $kos_2$ is the second Koszul matrix and $\gamma$ is a $1\times 5$ matrix of scalars. This accounts for five parameters for the codimension. The sixth comes from the fact that the subspace $M \in \GG$ depends only $\beta$ up to a scalar.

Hence, the parameter space $\MM$ has the expected codimension $8$ in $\GG_1$ and we should find a point in $\MM(\FF_3)$ in about $3^8\approx 6,500$ trials. In the 90s such a search took several hours; today it takes less than a minute.

Our search yields nine families of surfaces of degree 11 and genus 10.
One family consists of Enriques surfaces blown up at 6 points; the other 8 families
contain rational surfaces. They differ in their adjunction behavior, see the documentation of 
the function examplesOfSchreyerSurface in our package NongeneralTypeSurfaceInP4.
In the table of surfaces below, the last entry lists the components of the scheme $R_5=(X_5:X)$ residual to $X$ in the intersection of quintic hypersurfaces containing $X$. Each component is listed with the dimension and degree of $R_5$ and the dimension and degree of $R_5\cap X$. These components are either $6$-secant lines or planes that intersect $X$ in quartic curves. 

\medskip
\hskip -.2 cm
\begin{tabular}{l|c|c}
surface $X \subset \PP^4$ & extra syzygies & Components of $(X_5:X)$\\
\hline
    $Y(H_{min};2^1,1^5)$ & $2$ &$5\times (1,1,(0,6))$\\
    $\PP^2(18;6^5,5^5,2^1,1^4)$ & $2$&$6\times (1,1,(0,6))$\\
   $\PP^2(15;5^5,4^4,3^2,2^1,1^3)$ & $2$ &$7\times (1,1,(0,6))$\\
   \hline
   $\PP^2(16;6^1,5^6,4^3,2^2,1^3)$ & $2$& $3\times (1,1,(0,6)),2\times (1,1,(0,5)),$\\
   &&$ (2,1,(1,4))$ \\
   \hline
   $\PP^2(15;5^6,4^2,3^2,2^3,1^2)$ & $3$ & $2\times (1,1,(0,6)), 2\times (2,1,(1,4))$\\
   $\PP^1\times\PP^1((9,9);4^8,3^1,2^3,1^2)$ & $3$ & $3\times (1,1,(0,6)), (2,1,(1,4))$\\
   $\PP^2(13;4^7,3^4,2^2,1^2)$ & $3$ & $5\times (1,1,(0,6)), (2,1,(1,4))$ \\
   $\PP^2(13;5^1,4^5,3^4,2^4,1^1)$ & $4$ & $2\times (1,1,(0,6)),3\times (2,1,(1,4))$\\
   $\PP^2(12;4^4,5^5,2^6)$ & $5$ & $5\times (2,1,(1,4))$\\
\end{tabular} 

\smallskip\noindent
The surface $Y$ in the first case denotes a minimal Enriques surface. \medskip

\noindent
{\bf Lift to Characteristic 0.}
It remains to lift these surfaces to characteristic $0$. For this we note that
the scheme $\MM_k$ of modules with $k$ extra syzygies is defined over the integers and check that the fiber of
$\MM_k \to \Spec \ZZ$ over $(3)\in \Spec \ZZ$ is smooth of expected codimension $14+\dim \GG(2,k)$ in $\GG$ at the given point $M \in \MM_k(\FF_3)$ by a deformation computation, see the documentation of the function tangentDimension in our package NongeneralTypeSurfacesInP4. Hence, we can find a linear subspace in the Pl\"ucker space $P \subset \PP(\Lambda^{5}(H^0(\PP^4,\sO(2))^*))$ defined over $\ZZ$ that intersects $\MM_k(\FF_3)$ transversally at our point $M$. A component of $\MM_k\cap P$ contains an open subset of $\Spec \sO_L$ for a number field $L$ and a prime $\mathfrak p \in \Spec \sO_L$ with residue field $\sO_L/\mathfrak p\cong \FF_3$, such that mod $\mathfrak p$ the family reduces to our given surface over $\FF_3$. 

The generic point of that component yields a smooth surface $X_L$ defined over the number field $L$. By \cite{Sch96}, the surface $X_L$ has the same adjunction behavior as the surface over the finite field. Note that the degree of the number field equals the degree of the component of $\MM_k$ containing $M$. Perhaps with some effort one could determine the degree of $L$, but bounding the discriminant of the  number field $L$ seems to be out of reach. 
\end{example}

\begin{example}[Rational surfaces of degree $d=11 $ and sectional genus $\pi=11 $,  from \cite{GvBEL05}]

Graf v Bothmer, Erdenberger, and Ludwig discovered a rational surface of degree $d=11$ and sectional genus $\pi=11$ by searching over $\FF_2$.
The surface is the image of $\PP^2$ blown-up at $20$ points $p_1,\ldots,p_{20}\in \PP^2$ embedded by the linear system $H=9L-3E_1-\sum_{i=2}^{15} 2E_i -\sum_{j=16}^{20} E_j$, i.e., by the rational map defined by forms of in $H^0(\PP^2,\sI_\Gamma(9))$, where 
\[\sI_\Gamma= \sI_{p_1}^3\cap \bigcap_{i=2}^{15} \sI_{p_i}^2\cap \bigcap_{j=16}^{20} \sI_{p_j}.\]

Since $\chi(\sI_\Gamma(9))= \binom{9+2}{2}-\binom{3+1}{2}-14\binom{2+1}{2}-5=2$, 
we need $\h^1(\sO_X(H))=3$ and the points must lie in special position. By Proposition \ref{codimOfSpecialCollectionOfPoints}, the expected codimension of these points in $\Hilb_{20}(\PP^2)$ is $3\cdot 5=15$.
Thus, one might expect to find a surface in about $2^{15}\approx 32,000$ trials by choosing a collections of $1$, $14$ and $5$ points at random over $\FF_2$.

The resulting surface has Betti table 

\[ \begin{array}{rccccc}
       & 0 & 1 & 2 & 3 & 4\\
      \text{total:} & 1 & 8 & 13 & 8 & 2\\
      0: & 1 & . & . & . & .\\
      1: & . & . & . & . & .\\
      2: & . & . & . & . & .\\
      3: & . & 1 & . & . & .\\
      4: & . & 5 & 4 & . & .\\
      5: & . & 2 & 9 & 8 & 2
      \end{array}  \]

\noindent
{\bf Lift to Characteristic 0.} We use liaison theory to show the existence of the aforementioned surface in characteristic $0$ 
(see \cite{GvBEL05} for a deformation-theoretic argument showing how to  lift to characteristic $0$). 

The surface has two $6$-secant lines $L_1$ and $L_2$ spanning a hyperplane $H$. The surface $Y$ of degree $9$ linked to $X$ in a general $(4,5)$-complete intersection is singular precisely along $L_1$ and $L_2$. The hyperplane section $H\cap X$ is reducible with a degree-$8$ component of type $(3,5)$ on a quadric surface $Q$. The two $6$-secant lines $L_1$ and $L_2$ lie on $Q$ and are $5$-secants to the $(3,5)$-curve. Thus, a codimension-one space among the quintic hypersurfaces containing $X$ also contains $Q$. Hence, one may choose the $(4,5)$-linkage so that $Y$ is the union of $Q$ and a surface of degree $7$. Further analysis shows that this surface of degree $7$ is the union of a plane $P$ and a Bordiga surface $B$ (cf. Example~\ref{Bordiga}). The plane $P$ intersects $X\cup Q$ in a sextic plane curve. The surfaces $P$ and $B$ intersect along a line and two points, where this line is a $(-2)$-line on the Bordiga surface $B$. Such a line  occurs as the strict transform of a line through three of the points in $\PP^2$ that are blown up on the Bordiga surface (cf. Example \ref{Bordiga}). Reversing this linkage from $Y=Q\cup P\cup B$ yields a unirational construction of $X$.
\end{example}

\section{Tate resolutions}\label{tate}

The next idea for constructing surfaces in $\PP^4$ comes from the theory of Tate resolutions \cite{EFS03}, a framework that converts the cohomology groups of a sheaf on a projective space into a doubly infinite exact sequence of graded free modules over an exterior algebra. Not only does the Tate resolution of a sheaf encode its cohomology groups and those of its twists but its truncation also yields ``Beilinson's monad,'' a complex that constructs the sheaf from fundamental vector bundles such as the cotangent bundle $\Omega^1(1)$ twisted by $1$ and its exterior powers $\Omega^i(i)$. This section reviews the basics of the theory of Tate resolutions and provides examples illustrating how use them to construct surfaces in $\PP^4$.

\begin{theorem}{\rm \cite[Theorem 4.1]{EFS03}} Let $\sE$ be a coherent sheaf on $\PP^n=\PP(W)$, let $V=W^*$ denote the dual space, and let $E=\Lambda V$ denote its exterior algebra. There exists a doubly infinite exact sequence \[\TT= \ldots \to T^d \to T^{d+1} \to \ldots\] of graded free modules over $E$ with terms
\[T^d= \sum_{i=0}^n \Hom_K(E,H^i(\PP^n,\sE(d-i))).\]
 \end{theorem}   

Note that $\Lambda W= \omega_E=\Hom_K(E,K) \cong E(-n-1)$ is the free $E$-module generated by $\Lambda^{n+1} W$ with $E=\Lambda V$ acting via contraction.  Also, unlike rings and sheaves in previous sections we denote direct sums of $m$ copies of $E$ or $\Omega^i(i)$ by $mE$ and $m\Omega^i(i)$.

If $d \ge r$, where $r$ is the Castelnuovo-Mumford regularity of $\sE$, then the linear matrices with entries in $V \subset E$ defining the differentials in $\TT$ are the adjoints of the multiplication maps 
\[\mu_d\colon W\otimes H^0(\sE(d)) \to H^0(\sE(d+1)).\]
The fact that the compositions
\[W\otimes W \otimes H^0(\sE(d)) \to W \otimes H^0(\sE(d+1))\to H^0(\sE(d+2))\]
factor through $S_2W \otimes H^0(\sE(d)) \to H^0(\sE(d+2))$ implies that the adjoint matrices define a complex over $E$.

Thus, computing syzygies over the exterior algebra allows us to compute the cohomology groups of the coherent sheaf $\sE$. In particular, our cohomology table for the ideal sheaf~$\sI_X$ of a smooth surface $X \subset \PP^4$ coincides with the Betti table of the Tate resolution~$\TT(\sI_X)$.

One can recover Beilinson's monad \cites{OSS80,Bei78} for $\sE$ by applying the additive functor to $\TT(\sE)$
\[ \UU\colon \{\hbox{free graded $E$-modules}\} \to \{ \hbox{locally free sheaves on $\PP^n$}\} \] defined by
\[ \UU\colon \omega_E(i) \mapsto \Lambda^i U\]
and the contraction
\[\xymatrix{
 \omega_E(i) \ar@{|->}[r] \ar[d]^{\neg a} &\Lambda^i U \subset \Lambda^i W \otimes \sO\ar[d]^{\neg a}\\
 \omega_E(j) \ar@{|->}[r]  & \Lambda^j U \subset \Lambda^j W\otimes \sO \\}\]
by $a \in \Lambda^{i-j} V$, where $U\subset W\otimes \sO$ denotes the tautological rank-$n$ subbundle on~$\PP^n$ when identifying $\PP^n$ with $\GG(n,W)$.

\begin{theorem}[\cite{EFS03}]
\label{thm:Beilinson} 
Let $\sE$ be a sheaf on $\PP^n$. 
Then the complex $$\UU(\TT(\sE))$$ is a monad for $\sE$, i.e., 
$$H^i(\UU(\TT(\sE)))=0 \hbox{ for } i\not=0 \hbox{ and } H^0(\UU(\TT(\sE)))\cong \sE.$$
\end{theorem}
\noindent
Note that $\UU(\TT(\sE))$ forms a bounded complex because
$\Lambda^i U \not=0$ only for $0 \le i \le n$.
The proof of this theorem, presented in \cite{EES15}, is more conceptual in a broader context.

\begin{remark}
A similar functor computes $Rp_*(q^*\sE)$ for the incidence
flag 
\[\xymatrix{\FF \ar[d]^q \ar[r]_p  & \GG(r,W)\\
\PP^n \\}
\]
when we replace $U$ with the universal rank $r$ subbundle on the Grassmannian $\GG(r,W)$ of  $r$-dimensional subspaces of $W$ (c.f. \cite{ESW03})). 

If $\GG(n,W)\cong \PP(W)=\PP^n$, then we have $\Lambda^i U \cong \Omega^i(i)$ and the Beilinson monad is usually expressed in terms of the $\Omega^i(i)'s$. 
\end{remark}

\begin{example}[A surface of degree $d=8$ sectional genus $\pi=5$, $p_g=0$, and $q=1$ \cite{ADS98}]
If a smooth surface $X$ of degree $8$, sectional genus $5$, geometric genus~$0$, and irregularity $1$ in $\PP^4$ exists, then it is known to be an elliptic conic bundle (see, e.g., \cite{Ok86}). We demonstrate how to use Beilinson monads to construct such a surface.  

Note that $X$ has natural cohomology (see \cite[Proposition 2.2]{ADS98}). Therefore, the Tate resolution of $\sI_X$ takes the following form: 
\[ \begin{array}{rccccc ccccc}
        & -1 & 0 & 1 & 2 & 3 & 4 & 5 & 6   \\ \hline        
     -4: & 1 & . & . & . & . & . & . & .  \\
     -3: & 64 & 36 & 16 & 4 & . & . & . & .  \\
     -2: & . & . & . & 1 & 1 & . & . & .  \\
     -1: & . & . & . & . & 1 & 1 & . & .  \\
     0: & . & . & . & . & . & 6 & 26 & 66  \\
     \end{array}\]
Applying Theorem~\ref{thm:Beilinson} to $\sI_X(3)$ shows that $\sI_X(3)$ should arise as the cohomology of a monad 
\[
0 \rightarrow 4 \sO(-1) \oplus \Omega^3(3) \rightarrow \Omega^2(2) \oplus \Omega^1(1) \rightarrow \sO \rightarrow 0, 
\]
from which $\sI_X(3)$ can be interpreted as the degeneracy locus of a map from $4\sO(-1)$ to $\sG$ 
\[
0 \rightarrow 4\sO(-1) \rightarrow \sG \rightarrow \sI_X (3) \rightarrow 0, 
\]
where $\sG$ is a rank-$5$ torsion-free sheaf obtained as the cohomology of a monad
\[
0 \rightarrow \Omega^3(3) \rightarrow \Omega^2(2) \oplus \Omega^1(1) \rightarrow \sO \rightarrow 0.
\]
Since the adjoint line bundle $\sO_X(K_X+H)$ of $X$ is globally generated (see, e.g., \cite{Sommese79}), the torsion-free sheaf $\sG$ is actually locally free. This means that  the first map is injective and the second map is surjective as bundle maps. 

In~\cite{ADS98}, Decker, Sasakura, and the first author showed that if $\alpha \in V$ and $\beta \in \Lambda^2 V$ are generic, then the vertical concatenation of $\alpha$ and $\beta$ and the horizontal concatenation of $-\beta$ and $\alpha$ yield the required first and second maps of the monad. This construction results in the desired rank-$5$ vector bundle. 
The same paper further showed that the Betti table of the surface takes the following form: 
\[ \begin{array}{rccccc}
       & 0 & 1 & 2 & 3 & 4\\
      \text{total:} & 1 & 6 & 9 & 5 & 1\\
      0: & 1 & . & . & . & .\\
      1: & . & . & . & . & .\\
      2: & . & . & . & . & .\\
      3: & . & 6 & 4 & 1 & .\\
      4: & . & . & 5 & 4 & 1\\
      \end{array}  \]
\end{example}
   
\begin{example}[Rational surfaces of degree $12$ and sectional genus $13$]\label{AboRanestad}
A smooth rational surface $X$ of degree $12$ and sectional genus $13$ has a plausible Tate resolution for $\sI_X$ with shape
\[ \begin{array}{rccccc ccccc}
        & -1 & 0 & 1 & 2 & 3 & 4 & 5 & 6 & 7  \\ \hline        
     -4: & 1 & . & . & . & . & . & . & . &. \\
     -3: & 121 & 73 & 37 & 13 & . & . & . &. &.. \\
     -2: & . & . & . & . & 4 & 2 & . & .&.  \\
     -1: & . & . & . & . &  & 2 & 3 & .&.  \\
     0: & . & . & . & . & . & . & 5 & 29 &77  \\
     \end{array}\]
So, the Beilinson monad of $\sI_X(4)$ has shape
\[
   0\to4\Omega^3(3)\to 2\Omega^2(2)\oplus 2\Omega^1(1)\to 3\mathcal{O}\to 0, 
   \]
and hence the surface $X$ is completely determined by $m_{(2+2)\times 4 }:4E(1) \to 2E(2)\oplus 2E(3)$ and
$m_{3\times (2+2)}: 2E(2)\oplus 2E(3) \to 3E(4)$.

To construct these differentials, we start with the maps of linear forms $m_{2\times 4}:4E(1) \to 2E(2)$ and  $m_{3\times 2}:2E(3) \to 3E(4)$, and ask whether they can be extended by quadratic parts to maps
$m_{(2+2)\times 4}$ and
$m_{3\times (2+2)}$
which compose to zero.
In other words, we ask whether the space of homomorphisms
\[{\rm HomSpace}:=\Hom_E(\coker m_{2\times 4},\coker m_{3\times 2})\]
is nonzero.  In fact, 
$m_{3\times (2+2)}\circ m_{(2+2)\times 4}=0$
if and only if there is a map $\rho$ in HomSpace such that the diagram
\[\xymatrix{
 4E(1) \ar[r]^{m_{2\times 4}} \ar[d]_{-m'_{2\times 4}}& 2E(2)
 \ar[d]^{m'_{3\times 2}}\ar[r] & \coker  m_{2\times 4} \ar[d]^{\rho} \\
 2E(3) \ar[r]_{m_{3\times 2}} &3E(4)\ar[r] & \coker  m_{3\times 2}\\
}\]
commutes.

To compute the HomSpace from given matrices $m_{2\times 4}$ and $ m_{3\times 2}$, we extend them by adding indeterminate quadratic entries to $m_{(2+2)\times 4}$ and $m_{3\times (2+2)}$. We then require that their composition $ m_{3\times (2+2)}\circ m_{(2+2)\times 4}$ vanish.  
Consider the diagram
\[\xymatrix{
 4E(1) \ar[r]^{m_{2\times 4}} \ar[d]_{-m'_{2\times 4}}& 2E(2)\ar[dl]^{m_{2\times 2}}  \ar[d]^{m'_{3\times 2}}\\
 2E(3) \ar[r]_{m_{3\times 2}} &3E(4)\\
},\]
where $m_{2\times 2}: 2E(2)\to 2E(3)$ is a map defined by a matrix of linear forms.
The $20$-dimensional space of maps $m'_{3\times 2}$ that factor as $m_{3\times 2}\circ{m_{2\times 2}}$ is equivalent to the $0$-map in the HomSpace. There is a $140$-dimensional space of indeterminate $2$-form entries in the vertical maps in the diagram, while $$m_{3\times (2+2)}\circ m_{(2+2)\times 4}=0$$ imposes $120$ equations on their coefficients.  The dimension of the HomSpace is therefore the corank of this space of $120$ equations. 

    Now we fix the bottom horizontal matrix $m_{3\times 2}$ to correspond to the Hilbert-Burch matrix of the cubic scroll, i.e., 
    $$m_{3\times 2}^t=\begin{pmatrix}
        e_0& e_1 &e_3 \cr
        e_1 & e_2 &e_4
    \end{pmatrix}.
    $$
 The dimension of the space spanned by the equations $ m_{3\times (2+2)}\circ m_{(2+2)\times 4}=0$ depends solely on the top horizontal matrix $m_{2\times 4}$.

The entries of $m_{2\times 4}$ and $m_{3\times 2}$ are linear forms. Thus, each column of $m_{2\times 4}$ and each row of $m_{3\times 2}$ define a line in $\check \PP^4$. So, the columns of $m_{2\times 4}$ form a $\PP^3$ of lines, while the rows of $m_{3\times 2}$ form a $\PP^2$ of lines. Both of these $\PP^3$ and $\PP^2$ are embedded in $\GG(2,5)$ via the Pl\"ucker map:  
$$
\PP^2 \to \GG(2,5) \leftarrow \PP^3.
$$
The intersection points of the two images play a crucial role.
For a general choice of a matrix $m_{2\times 4}$, there are no intersection points. 

The dimension of the HomSpace is at least $h$ if there are $h$ intersection points \cite[Lemma 3.4]{AR06}. Furthermore, if the number of intersection points is finite, then there are at most seven intersection points. 

We observed that
when $h\geq 3$, the general map in HomSpace produces maps  $m_{(2+2)\times 4}$ and $ m_{3\times (2+2)}$, such that the homology of the composition is the ideal sheaf of a smooth surface. 

The varieties of matrices $m_{2\times 4}$ with 3 or 4 intersection points in $\GG(2,5)$ are unirational. 
Examples with 5 or 6 intersection points can be found by searching over a 
finite field. A special case arises if the $2\times 4$ matrix has a $2\times 2$ submatrix 
which also depends only on $e_0,e_1$, and $e_2$. In this case, we have two conics in the plane spanned by $e_0,e_1$, and $e_2$ 
intersecting at four points.  Adding one or two more intersection points 
in the Grassmannian $\GG(2,5)$ yields additional unirational components. To obtain a seventh intersection point, further searching is necessary. 

We have found seven families of surfaces, four of which are unirational.
These families differ in their adjunction behavior and in the number of 6-secants.
In the table below, the last three cases correspond to the special cases.\\

\begin{table}[ht]
\centering
\begin{tabular}{l|c|c}
surface $X \subset \PP^4$ & $h$ & $(X_5:X)$\\
\hline
    $\PP^2(12;4^2,3^9,2^3,1^7)$ & $6$& $(1,1,(0,6))$\\
   $\PP^2(12;4^3,3^6,2^6,1^6)$ & $5$&$2\times (1,1,(0,6))$\\
   $\PP^2(12;4^4,3^3, 2^9,1^5)$ & $4$&$3\times (1,1,(0,6))$ \\
   $\PP^2(12;4^5,2^{12},1^4)$ & $3$ &$4\times (1,1,(0,6))$\\
   $\PP^2(13;5,4^3,3^8,2,1^8)$ & $7$ & $(1,2,(0,11))$\\
   $\PP^2(13;5,4^4,3^5, 2^4,1^7)$ & $6$& $(1,1,(0,6)),(1,2,(0,11))$ \\
   $\PP^2(13;5,4^5,3^2,2^{7},1^6)$ & $5$& $2\times (1,1,(0,6)),(1,2,(0,11))$\\
\end{tabular}
\end{table}
The rightmost column lists the components of the scheme $R_5=(X_5:X)$ residual to $X$ in the intersection of quintic hypersurfaces containing $X$. Each component is listed with  its dimension and degree, along with the dimension and degree of $R_5\cap X$. These components are either $6$-secant lines or $11$-secant conics to $X.$

Note that Le Barz's 6-secant formula holds: The number of (-1)-lines plus the number of 6-secant lines is always eight.

We observed that the number of (-1)-lines equals the number of intersection points $h$ plus one.
We also observed that the 11-secant conic lies in the special plane spanned by
$e_0,e_1$, and $e_2$.\\

\noindent
{\bf Lift to characterstic 0.} 
The condition that the two linear matrices $m_{3\times 2}$ and $m_{2\times 4}$ define maps whose images intersect in $h$ points is a transversality condition in the Grassmannian. Thus, we can find a number field $L$
and matrices $m_{3\times 2}$ and $m_{2\times 4}$ defined over an open subset of $\Spec \sO_L$ containing a prime $\mathfrak p$ with residue field 
$\sO_L/\mathfrak p \cong \ZZ/p $
which reduce modulo $\mathfrak p$
to our given matrices over the finite field.
Since the dimension of the HomSpace is at least $h$ with and equals in our examples over a finite field, the dimension of the HomSpace over the number field is $h$ as well, and our families lift to characteristic zero.
\end{example}
    
\begin{example}[$d=12$, $\pi=13$, $p_g=1$, and $q=0$]\label{Abo surfaces}
If a regular smooth surface $X$ of degree $12$, sectional genus $13$ and Euler characteristic $2$ has natural cohomology, then the Tate resolution for $\sI_X$ takes the following form: 
\[ \begin{array}{rccccc ccccc}
        & -1 & 0 & 1 & 2 & 3 & 4 & 5 & 6 & 7  \\ \hline        
     -4: & 1 & . & . & . & . & . & . & . &. \\
     -3: & 122 & 74 & 38 & 14 & 1 & . & . &. &.. \\
     -2: & . & . & . & . & 3& 1 & . & .&.  \\
     -1: & . & . & . & . &  & 3 & 4 & .&.  \\
     0: & . & . & . & . & . & . & 4 & 28 &76  \\
     \end{array}\] 
To understand the differentials that determine this resolution, we focus on the two maps of linear forms $m_{1\times 3}:3E(1) \to E(2)$
and $m_{4\times 3}:3E(3) \to 4E(4)$ and ask whether they can be extended with quadratic parts to form maps
$m_{(1+3)\times 3}:3E(1) \to E(2)\oplus 3E(3)$
and 
$m_{4\times (1+3)}: E(2)\oplus 3E(3) \to 4E(4)$
that compose to zero. In other words, we ask whether the space of homomorphisms
\[{\rm HomSpace}:=\Hom_E(\coker m_{1\times 3},\coker m_{4\times 3})\]
is nonzero.

In fact, 
$m_{4\times (1+3)}\circ m_{(1+3)\times 3}=0$
if and only if there is a map $\rho\in {\rm HomSpace}$ such that the diagram
\[\xymatrix{
 3E(1) \ar[r]^{m_{1\times 3}} \ar[d]_{-m_{3\times 3}}& E(2)
 \ar[d]^{m_{4\times 1}}\ar[r] & \coker  m_{1\times 3} \ar[d]^{\rho} \\
 3E(3) \ar[r]_{m_{4\times 3}} &4E(4)\ar[r] & \coker  m_{4\times 3}\\
}\]
 commutes.

After fixing the top horizontal matrix as 
$$m_{1\times 3}=\begin{pmatrix} e_0 & e_1 & e_2    
\end{pmatrix},$$ 
the dimension of the space spanned by the equations $ m_{4\times (1+3)}\circ m_{(1+3)\times 3}=0$ depends only on the bottom horizontal matrix $m_{4\times 3}$, 

The matrix $m_{4\times 3}$ of linear forms in $e_0,\ldots,e_4$ has a  $\PP^3$ of rows, while $m_{1\times 3}$ has a $\PP^2$ of columns. As in Example~\ref{AboRanestad}, there is a linear relation among the equations $m_{4\times (1+3)}\circ m_{(1+3)\times 3} =0$ when the linear form of a column in $m_{1\times 3}$ and the linear forms in a row of the matrix $m_{4\times 3}$ together span only a $2$-dimensional space.

In a general matrix $m_{4\times 3}$, there are ten rows, corresponding to ten points in $\Pn^3$, where the entries are dependent. We refer to these as rank-two rows in~$m_{4\times 3}$.

In an extensive search, we found that nontrivial solutions to $m_{4\times (1+3)}\circ m_{(1+3)\times 3} =0$ may yield smooth surfaces only if at least four rank-two rows of $m_{4\times 3}$ contain a column of $m_{1\times 3}$. However, for a general choice of four or five such rank-two rows, the HomSpace is still $0$.

We found matrices $m_{4\times 3}$ whose four, five, or six rank-two rows contain a column of $m_{1\times 3}$.  In addition, we found matrices $m_{4\times 3}$ with a rank-one row and three rank-two rows that contain a column of $m_{1\times 3}$.

The dimension of the HomSpace varies across these examples.  Whenever it is nonzero, we choose a quadratic syzygy for the matrix $m_{4\times (1+3)}$.
In each case, $m_{4\times (1+3)}$ has five quadratic syzygies in addition to the syzygies $m_{(1+3)\times 3}$. A general choice extends $m_{(1+3)\times 3}$ to a matrix $m_{(1+3)\times (1+3)}$ 
which, together with $m_{4\times (1+3)}$, defines the maps of the Beilinson monad 
\[
\sO(-1)\oplus 3\Omega^3(3)\to \Omega^2(2)\oplus 3\Omega^1(1)\to 4\mathcal{O}.
\]
We then check that the homology of the composition is the ideal sheaf $sI_X(4)$ of a smooth surface.

The smooth surfaces turn out to be either nonminimal K3 or elliptic surfaces.
The canonical divisor $K$ has degree $H\cdot K=12$ and self-intersection number $K^2=-6$. Hence, in the K3 case, the canonical divisor is the sum of six $(-1)$-curves. The degrees of these curves form a partition of $12$ into $6$ parts.
We find smooth examples $X$ of nonminimal K3 surfaces over finite fields whose canonical divisor defines all but two partitions of $12$ into $6$ parts, as listed in the table below. 
Here, the number $N$ counts the number of rank two rows of $m_{4\times 3}$ that contain a column in $m_{1\times 3}$. The next entry is the dimension of the HomSpace. The third entry is the partition of $12$ formed by the degrees of the $(-1)$-components of the canonical divisor $K$ of $X$. The fourth and last entry $(X_5: X)$ lists the components of the scheme $RX_5=(X_5:X)$ residual to $X$ in the intersection of quintic hypersurfaces containing $X$. Each component is listed with  the dimension and degree of $RX_5$, as well as the dimension and degree of $RX_5\cap X$.  
\medskip

\begin{tabular}{l|c|c|c}
$N$ & {\rm dim HomSpace} & Canonical divisor & Components of $(X_5: X)$  \\
\hline
 $  4$ &   $1$   &  $1,2,2,2,2,3$ & $6\times (1,1,(0,6)),(0,1)$\\
 $ 5$&    $1$   &  $1,1,2,2,3,3$ & $5\times (1,1,(0,6))$\\ 
 $6$ &   $ 1$   &  $1,1,1,3,3,3$ & $4\times (1,1,(0,6)),(1,4,(0,21))$\\
 $ 4$ &   $2$   &  $1,1,2,2,2,4$ & $5\times (1,1,(0,6)), (1,2,(0,11)), (0,1)$ \\
 $  6$ &   $2$   &  $1,1,1,2,3,4$ & $4\times (1,1,(0,6)),(1,2,(0,11))$\\
 \hline
  $ 5$ &   $3$   &  $1,1,1,2,2,5$ & $5\times (1,1,(0,6)), (1,1,(0,5))$,\\
  &&&$(2,1,(1,5),(0,1))$ \\ \hline
  $ 6$ & $4$   &  $1,1,1,1,4,4$ & $3\times (1,1,(0,6)),2\times (1,2,(0,11))$\\ \hline
$6$ &   $4$   &  $1,1,1,1,2,6$ & $4\times (1,1,(0,6)), 2\times (1,1,(0,5)),$\\
&&&$(2,1,(1,5),(0,1))$ \\ \hline

 $ 7$ & $5$   &  $1,1,1,1,1,7$ & { $3\times (1,1,(0,6)), 3\times (1,1,(0,5)),$}  \\ 
 &&& $(2,1,(1,5),(0,1))$\\ 
\end{tabular} \\

The two partitions of $12$ of length six for which we do not find examples are
$(2,2,2,2,2,2)$ and $(1,1,1,1,3,5).$

\smallskip
Additionally, we find examples of nonminimal elliptic surfaces with similar data. Here, the canonical divisor has six $(-1)$-curves as components, in addition to an elliptic curve. The degrees of these components are listed as above with the degree of the elliptic curve specified at the end. 

\medskip
\begin{tabular}{l|c|c|l}
$N$ & {\rm dim HomSpace} & Canonical divisor & Components of $(X_5: X)$  \\
\hline
     $7$ & $6$   &  $1,1,1,1,2,2,4$ & $3\times (1,1,(0,6)),(2,2,(1,8,g=9))$\\ 
\end{tabular}

\medskip

\noindent
{\bf Lift to characteristic 0.} Consider two estimates for the dimension of the Hilbert scheme of these surfaces. The first estimate considers for each family of minimal K3 surfaces obtained by blowing down the exceptional curves. Each such family belongs to a $19$-dimensional family of K3 surfaces with a given polarization. The choice of six points to blow up yields a $19+12$-dimensional total family, in which our surfaces form a sublocus of codimension at most $3s=15$, because our surfaces have speciality $s=3$. Accounting for the automorphisms of $\Pn^4$, we see that our surfaces form a family of dimension at least $40$.  
On the other hand, the tangent space of the set of isomorphism classes of monads yielding our surfaces has dimension $40$, so the dimension of our families over finite fields is at most $40$. 

The family of monads is defined over $\ZZ$, and hence so is the total family of ideal sheaves of surfaces. Over $(19)\in \Spec \ZZ$, there is a smooth surface, and the tangent space at the surface is $40$-dimensional, while the dimension of the family is at least~$40$. Therefore, the total family is smooth at the given surface defined over $\ZZ/(19)$, and hence so is the generic fiber.  
\end{example}
\begin{remark} For the nonminimal K3 surfaces with the canonical divisor possessing the $(-1)$-partition  $(1,1,1,3,3,3)$ or $(1,1,1,1,1,7)$, our constructions provide an alternative proof of existence in characteristic $0$.  In the case of the $(-1)$-partition $(1,1,1,3,3,3)$, we may argue as in the case of surfaces in Example \ref{AboRanestad}. The family of monads are defined over $\ZZ$, and the condition for the desired HomSpace is that the  matrix $m_{4\times 3}$ has six rank-two rows that contain a column of $m_{1\times 3}$. This condition is clearly transversal in the space of monads, so the locus of monads whose homology is the ideal of a surface is smooth. Since the surface is smooth at some $(p)\in \Spec \ZZ$, so is the corresponding surface over the generic point.  

For the surfaces with $(-1)$-partition $(1,1,1,1,1,7)$, we find a family of matrices $m_{4\times 3}$ by considering their adjoint $5\times 3$-matrices $M$ on $\Pn^3$.
In the  matrix $M$, the $2\times 3$ submatrix $N$ formed by the two last rows has rank one along three lines through a point $p$ that span the $\Pn^3$. Furthermore, the matrix $M$ has six additional rank-two points, one on each of the three lines through $p$. Thus, the rank-two locus of $M$ shares a scheme of length seven with the rank-one locus of $N$. The matrix $m_{4\times 3}$ adjoint to $M$ then fits into a monad with $5$-dimensional HomSpace and a general map in HomSpace gives a smooth surface over a small $p\in \ZZ$. The closed conditions involved in the construction of $M$ are clearly well defined over $\ZZ$; therefore, so are the monads. We may therefore conclude, as above, that the homology of a monad defined over $\ZZ$ is the ideal sheaf of a surface that is smooth $\mod p$ surface for $(p)\in \Spec  \ZZ$; therefore smooth over $\QQ.$
\end{remark}
\begin{remark} We suspect that the nonminimal elliptic surfaces in Example \ref{Abo surfaces} also lift to char $0$, but we leave this as an open problem.
\end{remark}
\begin{example}[$d=13$, $\pi=16$, $p_g=2$, and $q=0$]\label{linked Abo surfaces}
Let $X$ be a surface in Example~\ref{Abo surfaces} such that $RX_5$ has no projective plane as a component. This means the canonical divisor of $X$ corresponds to one of the following partitions of 12 into 6 parts:  (1) $(1,2,2,2,2,3)$, (2) $(1,1,2,2,3,3)$, (3) $(1,1,1,3,3,3)$, (4) $(1,1,2,2,2,4)$, (5) $(1,1,1,2,3,4)$, and (6) $(1,1,1,1,4,4)$. A general complete intersection of two quintics containing such an $X$ links $X$ to a smooth surface $Y$ of degree $13$ and sectional genus $16$ (cf. Remark \ref{linkage}). We show that if its canonical divisor corresponds to (2), (3), (5), or (6), then $Y$ is an elliptic surface by determining the type of its minimal model. 

By liaison, we obtain the linear system $|K_Y|=|5H-(X\cap Y)|$ on the surface $Y$. Since $h^0 (\sI_X(5))=4$, we deduce that $p_g(Y)=h^0(\sO_Y(K_Y))=2$ and that $|K_Y|$ is a pencil with the curve $C_5$ as a fixed component. Any $6$-secant line $L$ to $X$ is a component of $C_5$, and
$L\cdot K_Y=L\cdot(5H-(X\cap Y))=5-6=-1$, so $L$ is a $(-1)$-line on $Y$. Moreover, via liaison, $\chi(\sO_Y)=3$, which shows that $q = h^0(\sO_Y(K_Y)) = 0$ and that $K_Y^2=-5$ by the double point formula.  

(2) If $C_5$ consists of five $6$-secant lines to $X$, then these five lines are the fixed component of any curve in $|K_Y|$.  Blowing down these lines yields a minimal surface $Y_0$ with $K_{Y_0}^2=0$. In particular, $Y_0$ is a minimal elliptic surface.  

(3) If $C_5$ consists of four $6$-secant lines to $X$ and a $21$-secant rational normal quartic curve $D$, then $D\cdot K_Y=D\cdot(5H-(X\cap Y))=20-21=-1$, so $D$ is a $(-1)$-curve on $Y$, and the four $6$-secant lines to $X$ together with $D$ form the fixed component of any curve in $|K_Y|$. Blowing down these components yields a minimal surface $Y_0$ with $K_{Y_0}^2=0$, and hence $Y_0$ is a minimal elliptic surface. 

(5) If $C_5$ consists of four $6$-secant lines and a $11$-secant conic $D$ to $X$, then $D\cdot K_Y=D\cdot(5H-(X\cap Y))=10-11=-1$, so $D$ is a $(-1)$-curve on $Y$, and the four $6$-secant lines to $X$ together with $D$ form the fixed component of any curve in $|K_Y|$. Blowing down these components yields a minimal surface $Y_0$ with $K_{Y_0}^2=0$, and hence $Y_0$ is a minimal elliptic surface.

(6) Assume that $C_5$ consists of three $6$-secant lines and two $11$-secant conics to $X$. So, the components of $C_5$ are three $(-1)$-lines and two $(-1)$-conics that form the fixed component of any curve in $|K_Y|$. Like above, since $K_Y^2=-5$, the minimal surface $Y_0$ obtained by blowing down these $(-1)$-curves has $K_{Y_0}^2=0$, and hence it is a minimal elliptic surface. 

If the canonical divisor of $X$ corresponds to (1) or (4), then one obtains a linked surface $Y$ with six $(-1)$-lines, or five $(-1)$-lines and one $(-1)$-conic, respectively. Since $K_Y^2=-5$, the minimal surface $Y_0$ obtained by blowing down these exceptional curves has $K_{Y_0}^2=1$, so it is a surface of general type.
\end{example}

\section{Constructions using Moore matrices}
This section showcases examples of abelian and bielliptic surfaces in $\PP^4$ and details their construction methods. These are among the very few known examples of irregular surfaces in $\PP^4$. 

The constructions, described here, utilize syzygy techniques, as outlined in Section~\ref{Section: Hilbert Burch}. A crucial step involves constructing their Hartshorne-Rao modules through Moore matrices, which play a structural role in the geometry of the Horrocks-Mumford bundle, essentially the unique indecomposable vector bundle of rank $2$ known to exist on $\PP^4$ in characteristic $0$. 
\begin{example}[Horrocks-Mumford bundle] 
Let $V$ be a five-dimensional vector space over $\CC$ with basis $\{e_0,\dots,e_4\}$. Consider the $5\times 2$ matrix
$$\varphi =\begin{pmatrix}
    e_1\wedge e_4 & e_2\wedge e_0 & e_3\wedge e_1 & e_4\wedge e_2 & e_0\wedge e_3 \cr
    e_2\wedge e_3 & e_3\wedge e_4 & e_4\wedge e_0 & e_0\wedge e_1 & e_1\wedge e_2 \cr
\end{pmatrix}.$$
It defines one differential in a Tate resolution of shape 
$$\begin{array}{r ccccccc ccccc}
     &-5 &-4 & -3 & -2 & -1 & 0 & 1 & 2 & 3 & 4 & 5 \\\hline 
     -4:&100 & 35 & 4 & . & . & . & . & . & . & . & . \\
     -3:& . & 2 & 10 & 10 & 5 & . & . & . & . & . & . \\
     -2:& . & . & . & . & . & 2 & . & . & . & . & . \\
     -1:& . & . & . & . & . & . & 5 & 10 & 10 & 2 & . \\
     0: & . & . & . & . & . & . & . & . & 4 & 35 & 100 \\
     \end{array}$$
The cohomology of the corresponding monad
\[ 0 \to 5\sO(-1) \to 2 \Omega^2(2) \to 5\sO \to 0\] 
is a vector bundle because it has bounded middle cohomology (c.f., \cite[Example 7.3]{EFS03}). 
It is the famous Horrocks-Mumford bundle $\sF_{HM}$. Its Chern polynomial is
\[ c_t(\sF_{HM})=1-t+4t^2.\]
Apart from tensoring with a line bundle and considering endomorphisms of $\PP^4$, the Horrocks-Mumford bundle remains the only known indecomposable rank-2 vector bundle on $\PP^4$ defined over $\CC$ (c.f., \cites{HM73,DS86,DS87}).
\end{example}

\begin{example}[Abelian surfaces]
The Chern polynomial of $\sF_{HM}(3)$ is $$c_t(\sF_{HM}(3))=1+5t+10t^2.$$ A general section of $\sF_{HM}(3)$ vanishes along a smooth surface $X$ of degree $10$, resulting in a  short exact sequence
\[0 \to \sO \to \sF_{HM}(3) \to \sI_X(5) \to 0.\]
Thus, the Tate resolution of $\sI_X$ 
takes the following form: 
$$\begin{array}{r ccccc ccccc}
       & -1 & 0 & 1 & 2 & 3 & 4 & 5 & 6 & 7 \\\hline 
     -4: & 1 & . & . & . & . & . & . & . & . \\
     -3: & 80 & 45 & 20 & 5 & 1 & . & . & . & . \\
     -2: & . & . & . & 2 & . & . & . & . & . \\
     -1: & . & . & . & . & 5 & 10 & 10 & 2 & . \\
     0: & . & . & . & . & . & . & 3 & 30 & 85 \\
     \end{array}$$
This implies that $X$ has $\pi=6$, $p_g=1$, and $q=2$, from which $K^2=0$ follows. This surface is known to be a minimal abelian surface with a $(1,5)$-polarization and was first found by Commessati \cite{Com19} as a Jacobian surface.

Le Barz's formula \ref{LeBarz} gives $N_6=25$. These twenty five $6$-secant lines are the famous Horrocks-Mumford lines. The surface $X'$ linked to $X$ in a $(5,5)$ complete intersection has $d = 15$, $\pi = 21$, $p_g=1$, and $q=2$, and hence $K_{X'}^2=-25$. The Horrocks-Mumford lines are the twenty five $(-1)$-lines of $X'$, as was discussed in Example \ref{linked Abo surfaces}, and hence $X'$ is an abelian surface
blown-up in 25 points.
\end{example}
Elliptic normal curves and abelian surfaces in $\Pn^4=\PP(W)$ have automorphisms that lift to projective linear transformations of the ambient space.  The relevant group of automorphisms is the Heisenberg group $H_5$ of level $5$. If $\{x_0,\dots ,x_4\}$ is its dual basis, then $H_5$ is generated by \[ \sigma: e_i\mapsto e_{i-1}, \quad \tau: e_i\mapsto e^{2\pi i/5}e_i
\]
as a subgroup of $\SL(V)$. 
To study equations of abelian surfaces and other varieties invariant under the action of $H_5$, Moore \cite{Mo85} (see also \cite{ADHPR97}) introduced the $5\times 5$ matrices called {\it Moore matrices}
\[
M(y,x)=(y_{3i-3j}x_{3i+3j})_{i,j\in \ZZ_5},  
\]
parameterized by points $y\in \Pn^4(y).$ 
Together with matrices 
\[ L(z,x)=(z_{i-j}x_{2i-j})_{i,j\in \ZZ_5}. 
\]
parameterized by points $z\in \Pn^4(z)$, they may be used to build the maps of a Beilinson monad for varieties invariant under the action of $H_5$. Multiplication by $-1$ on $H_5$-invariant elliptic normal curves and abelian surfaces lifts to an involution 
\[\iota: e_i\mapsto e_{-i}
\] on the ambient space. 
Notice that $\iota$ fixes the two linear subspaces $\Pn^1_{-}$, defined by $x_0=x_1+x_4=x_2+x_3=0$, and $\Pn^2_{+}$, defined by $x_1-x_4=x_2-x_3=0$.
If $p\in \Pn^1_{-}(y)$, then $M(p,x)$ is skew-symmetric.

\begin{example}
The first Hartshorne-Rao module $H^1_*(\sF_{HM})$ of $\sF_{HM}$ has Hilbert function $(5,10,10,2,0,\ldots)$, and it can be interpreted as the cokernel of a $5\times 15$ linear matrix consisting of three blocks of Moore matrices 
\[
m_{5\times 15}=(M(p_1,x), M(p_2,x),M(p_3,x)),
\]
where the $p_i$ span $\Pn^2_{+}(y)$ (see \cite{De90, ADHPR97}).     
\end{example}

\begin{example}
Let $G_5$ be the subgroup of $\SL(V)$ generated by $H_5$ and $\iota$.
If $p\in \Pn^1_{-}(x)$, then the $4\times 4$ Pfaffians of $M(y,p)$ generate the ideal of a $G_5$-invariant elliptic curve in $\Pn^4(y)$ with origin at $p$.

If $p\in \Pn^1_{-}(x)$, then the $4\times 4$ minors of $L(\iota(z), p)$ define a $G_5$-invariant elliptic quintic scroll in $\Pn^4(z).$

The determinants of $M(y,x)$ for $y\in \PP^4$ define the Horrocks-Mumford quintics (see \cites{BHM87,Au88}).
\end{example}

\begin{example}
A bielliptic surface $X_b$ of degree $10$ was first discovered by Serrano \cite{Ser90}. Later, an alternative construction method for this surface was introduced using Moore matrices in \cite{ADHPR97}. 

The Tate resolution of $\sI_{X_b}$ takes the following form: 
$$\begin{array}{r cccc ccccc}
       & -1 & 0 & 1 & 2 & 3 & 4 & 5 & 6 & 7\\\hline
      -4: & 1 & . & . & . & . & . & . & . & .\\
      -3: & 80 & 45 & 20 & 5 & . & . & . & . & .\\
      -2: & . & . & . & 1 & . & . & . & . & .\\
      -1: & . & . & . & . & 5 & 10 & 10 & . & .\\
      0: & . & . & . & . & . & . & 1 & 30 & 85
      \end{array}$$
Its first Hartshorne-Rao module $N^b$ has the following $5 \times 15$ presentation matrix: 
\[
(M(p_1,x), M(p_2,x),M(p_3,x)), 
\]
where the $p_i$ are chosen on a $G_5$-invariant elliptic curve $E\subset \Pn^4(y)$. If $\tau_1,\tau_2,\tau_3 \in E\cap \Pn^2_+$ are the nontrivial $2$-torsion points on $E$, then we may choose $p_1=\tau_1$, $p_2=\tau_2$, and $p_3=\tau_3+\rho$, where $\rho$ is a nontrivial $3$-torsion point on $E$. 

The Beilinson monad for $\sI_{X_b}(3)$ yields an exact sequence
\[
0\to 5\mathcal{O}(-1)\oplus \Omega^3(3)\to Syz_1(N^b(1))\to \sI_{X_b}(3)\to 0 
\]
\end{example}

\begin{example}
A nonminimal bielliptic surface $X_{b'}$ of degree $15$ was  constructed using Moore matrices in \cite{ADHPR97}. The Tate resolution of $\sI_{X_{b'}}$ takes the form
$$\begin{array}{r cccc ccccc}
       & -1 & 0 & 1 & 2 & 3 & 4 & 5 & 6 & 7\\\hline
      -4: & 1 & . & . & . & . & . & . & . & .\\
      -3: & 170 & 105 & 55 & 20 & . & . & . & . & .\\
      -2: & . & . & . & 1 & 10 & 10 & 5 & . & .\\
      -1: & . & . & . & . & . & . & . & . & .\\
      0: & . & . & . & . & . & . & 1 & 15 & 50
      \end{array}$$

Form a $5\times 15$ matrix by horizontally concatenating three Moore matrices as follows: 
\[
m^{b'}_{5\times 15}=(M(p_1,x), M(p_2,x),M(p_3,x)),
\]
where the $p_i$ lie on a $G_5$-invariant elliptic curve $E\subset\Pn^4(y)$.  More precisely, if $\tau_1,\tau_2,\tau_3 \in E\cap \Pn^2_+$ are the nontrivial $2$-torsion points on $E$, then we choose $p_1=\tau_1+\rho$, $p_2=\tau_2+\rho$  and $p_3=\tau_3+\rho$, where $\rho$ is a nontrivial $3$-torsion point on $E$.

The matrix $m^{b'}_{5\times 15}$ is the presentation matrix of an artinian module $N'$. The syzygy bundle $Syz_2((N')^\vee (3))$ has rank $21$ and admits a unique sheaf monomorphism 
\[ \phi:20\mathcal{O}(-1)\to Syz_2((N')^\vee (3)), 
\]
and $\sI_{X_{b'}}(3)$ is obtained as the cokernel of $\phi$.
\end{example} 

\begin{example}[Popescu surface,\cite{Pop93}] In his thesis,  Popescu constructed a non-minimal abelian surface of degree 15 whose cohomology table is
$$
\begin{array}{r ccccc ccccc}
      & -1 & 0 & 1 & 2 & 3 & 4 & 5 & 6 & 7 & 8\\\hline
     -4: & 1 & . & . & . & . & . & . & . & . & .\\
     -3: & 170 & 105 & 55 & 20 & 1 & . & . & . & . & .\\
     -2: & . & . & . & 2 & 10 & 10 & 5 & . & . & .\\
     -1: & . & . & . & . & . & . & . & . & . & .\\
      0: & . & . & . & . & . & . & 1 & 15 & 50 & 115
     \end{array}
     $$
The trick is to consider a module $N$ with Hilbert function $(5,10,10,2)$ as the cokernel of a
$5\times 15$ matrix 
$$(M(p_1,x),M(p_2,x),M(p_3,x))$$
where $p_1, p_2$ span $\PP^1_-$ and $p_3$ lies in $\PP^2_+$. 
\end{example}

\section{Summary}
We know of $79$ components of the Hilbert scheme corresponding to
non-degenerate surfaces of Kodaira dimension $<2$.
The distribution into birational classes is as follows: \medskip

\begin{tabular}{|c|c|c|c|c|c|c|c|}\hline
deg  &rational   & ruled $q>0$ &  Enriques & K3 & bielliptic & abelian & elliptic\\ \hline
$\le 8$ & 7 & 2 & & 3& & &2 \\
9 & 2 & & 1& 1& & &1 \\
10& 3 & & 1& 2& 1&1 &2 \\
11&12 & & $2^* $& 5& & &1 \\
12& 7 & & & 10 & & & 4\\
13  & & & $2^{**} $ & 1& & & 4\\
14 &  & & & 1& & & \\
15 &  & & & & 1&2 & \\ \hline
 & 31 & 2 & $\geq 4$ &23 & 2 & 3 & 14 \\ \hline
\end{tabular} \medskip

\noindent
$2^*$ indicates that one of the families is locally incomplete. Thus it might lie on the boundary of the other family. Compare the functions enriqueSurfaceD11S10 and specialEnriquesSchreyerSurface in our package  NongeneralTypeSurfaceInP4 \cite{ARS}.\\
$2^{**}$ indicates that there are two irreducible families, but one might lie on the boundary of the other.\\

For each of these families, our package includes a function producing examples over a finite prime field $\ZZ/p$. When the corresponding component of the Hilbert scheme is unirational (most are) these examples are reductions mod $p$ of examples defined over an open part of $\Spec \ZZ$.

\section{Open problems}\label{problems}

\begin{question}
 Is there an upper bound on the degree of a nonminimal smooth surfaces in $\Pn^4$, or on smooth surfaces with a $(-1)$-line? 

 Bauer, in \cite{Bau95}, showed that if $X$ is an inner projection from a smooth point of a surface in $\Pn^5$, then the degree is at most $9$.

 The surface $Y$ linked in a (5,5) complete intersection to a non-special Alexander surface $X$, is a non-minimal surface of general type of degree 16 with a unique (-1)-curve which coincides with the 6-secant line of $X$; see the documentation for the function linkedNonspecialAlexanderSurfaceD16 in our package \cite{ARS}. 
\end{question} 

\begin{question}
 Is the irregularity of a smooth surface in $\PP^4$ bounded by $2$?
\end{question} 
  Abelian surfaces are examples of smooth surfaces in $\Pn^4$ with irregularity $2$ (cf. \cite{Com19}). 

  The maximal degree of currently known smooth rational surfaces in $\Pn^4$ is $12$ (cf.\cite{AR06}).
\begin{problem}
  Find a smooth rational surface in $\Pn^4$ of degree bigger than twelve.  Find the maximal degree.
\end{problem} 
Decker and Schreyer \cite{DS00} showed that an upper bound for the degree of smooth rational surfaces in $\Pn^4$ is $52$. Ellia \cite{Ellia00} showed that the upper bound is 12 if the surface has a pencil of rational curves of degree at most $4$.

\begin{problem}
  Find a smooth nongeneral-type surface in $\PP^4$ of degree bigger than 15.
\end{problem}

Among the rational surfaces  $X=\PP^2(a;a_1,..,a_r) \subset \PP^4$ of speciality $s>1$, only one family has a known description of the special collection of $r$ points in $\Pn^2$ blown up to $X$ (see Example \ref{second special surface}).
\begin{problem}
  Describe the special collection of $r$ points in $\Pn^2$ for all rational surfaces with $s>1$.
\end{problem}
    \begin{problem} The seven families of rational surfaces in Example \ref{AboRanestad} are distinguished by the number  $n$ of common rows in $m_{3\times 2}$ and columns in $m_{2\times 4}$.  These rows/columns define $n$  planes in $\Pn^4$, and the number of $(-1)$-lines  on the corresponding rational surface is $n+1$.
    Give a theoretical proof of this fact.
        \end{problem}

\begin{problem}
Are there further families of surfaces as in Example \ref{schreyer surfaces},
Example \ref{AboRanestad}, and Example \ref{Abo surfaces}?

 In particular, do there exist a non-minimal K3 surface of degree $12$ and sectional genus $13$, as in Example \ref{Abo surfaces}, whose canonical divisor consists of six $(-1)$-conics? 
 Such surface is expected to have seven $6$-secant lines.

Does there exists a nonminimal K3 surface of degree $12$ and sectional genus $13$, as in Example \ref{Abo surfaces}, whose canonical divisor consists of  four $(-1)$-lines and $(-1)$-curves of degrees $3$ and $5$.    
\end{problem}

\begin{problem} Prove that the elliptic surfaces in Example \ref{Abo surfaces}
can be lifted to characteristic $0$.  
\end{problem}

\begin{problem}
A plausible table for a smooth rational surface of degree $d=14$ and sectional genus $\pi=18$ is shown below.

\[ \begin{array}{rccccc ccccc}
        & -1 & 0 & 1 & 2 & 3 & 4 & 5 & 6 & 7 &8 \\ \hline        
     -4: & 1 & . & . & . & . & . & . & . &. \\
     -3: & 153 & 94 & 49 & 18 & . & . & . &. &. &. \\
     -2: & .    & . & . & .   & 7 & 6 & 1 & .&.  \\
     -1: & .    & . & . & .   & . & . &3 & . & .&.  \\
     0: & . & . & . & . & . & . & . & 17 &56 &126  \\
     \end{array}\]  
\end{problem}
Such a surface has $K^2=-16$, $H.K= 20$ and Le Barz's $N_6=19$. It could be the blow-up of $\PP^2$ at a total of $16+9=25$ points. An abstract surface of this kind would depend on $2\cdot 25-8=42$ parameters. On the other hand imposing the condition that $h^0(\sO(H))=5$ imposes at most $5\cdot 7 =35$ conditions by 
Proposition \ref{codimOfSpecialCollectionOfPoints}. So, the existence appears feasible.

A step forward would be to prove or disprove the existence of smooth
rational surfaces of this type.

\begin{problem}[Find a further rank 2 vector bundle on $\PP^4$]
Up to an automorphism of $\PP^4$ and tensoring with a line bundle, the Horrocks-Mumford bundle remains the only known indecomposable rank 2 vector bundle defined over $\CC$ (cf. \cites{DS86,DS87}).

Using random search over finite fields one might hope to find further bundles: first discover a rank-2 vector bundle over a finite field and subsequently lift it to a bundle defined over a number field.

As a proof of concept we have written a search routine which over $\FF_2=\ZZ/2$ discovers the Horrocks-Mumford bundle up to an automorphism of $\PP^4$ define over $\overline \FF_2$. This is a search in codimension 20, and succeeds in about 3 hours on average. \medskip

A tempting case is the search for a rank 2 vector bundle with Chern classes  $c_1=0$ and $c_2=11$. The Tate resolution of such a bundle could take the form 

$$\begin{array}{r ccccccccc cccccc}
     &-7&-6 &-5 &-4 & -3 & -2 & -1 & 0 & 1 & 2 & 3 & 4 & 5 & 6 & 7 \\ \hline 
     -4:&35&. &.& . & . & . & . & . & . & . & . & . & . & .&.\\
     -3:&45&80& 84 & 69 & 45 & 20 & .&. & . & . & . & . & . & . &. \\
     -2:&.& . & . & . & . &. &11 & 11 & . & . & . & . & . & .&.  \\
     -1:& .&.&. & . & . & . & . & . & 20 & 45 & 69 & 84 & 80 &45 & .\\
     0: & .& .&. & . & . & . & . & . & . & . & . & . &.& 35 & 176 \\
     \end{array}$$
The central $11\times 11$ matrix of linear forms is skew symmetric.
So far we have found $11\times 11$ matrices that yield the desired
complex in homological degrees $-2, \ldots, 1$. However, our examples have unbounded "middle cohomology" and hence do not correspond to a vector bundle.
\end{problem}

\begin{problem}
Most of the components of the Hilbert scheme of smooth nongeneral type surfaces that we discovered are unirational, and all are uniruled due to the action of $\PGL(5)$.
Decide whether some components are not unirational. Candidates are some components of the surfaces in Examples \ref{schreyer surfaces}, \ref{AboRanestad}, and \ref{Abo surfaces}.
\end{problem}

\end{document}